\pdfoutput=1 
\documentclass[journal]{IEEEtran}
\usepackage{amssymb,epsfig,color,cite,amsmath,amsfonts,mathrsfs,algorithm,algorithmicx}
\usepackage{epstopdf}
\usepackage{datetime,fancyhdr}
\usepackage{algpseudocode}
\usepackage{arydshln}
\usepackage{multirow}
\usepackage{amsthm,bm}
\usepackage{float}
\usepackage{threeparttable}
\usepackage{times} 
\usepackage{epsfig}
\usepackage{graphicx}
\usepackage{latexsym}
\usepackage{subfigure}

\newtheorem{lemma}{Lemma}[section]
\newtheorem{theorem}{Theorem}[section]
\newtheorem{corollary}{Corollary}[section]
\newtheorem{remark}{Remark}[section]

\newtheorem{proposition}{Proposition}[section]
\newtheorem{assumption}{Assumption}[section]

\begin{document}
%
\title{ Distributed stochastic projection-free solver for constrained optimization}
%
%
%

\author{Xia Jiang,
        Xianlin~Zeng,~\IEEEmembership{Member,~IEEE,}
        Lihua~Xie,~\IEEEmembership{Fellow,~IEEE,}\\
        Jian~Sun,~\IEEEmembership{Member,~IEEE,}
        and ~Jie~Chen,~\IEEEmembership{Fellow,~IEEE} 
\thanks{This work was supported by  the National Natural Science Foundation of China (Nos. 61925303, 62088101, 62073035) and China Scholarship Council 202006030072. \emph{(Corresponding author: Jian Sun.)}}
\thanks{X. Jiang (jiangxia@bit.edu.cn) and J. Sun (sunjian@bit.edu.cn) are with Key Laboratory of Intelligent Control and Decision of Complex Systems, School of Automation, Beijing Institute of Technology, Beijing, 100081, China, and also with the Beijing Institute of Technology Chongqing Innovation Center, Chongqing  401120, China}
\thanks{X. Zeng (xianlin.zeng@bit.edu.cn) is with Key Laboratory of Intelligent Control and Decision of Complex Systems, School of Automation, Beijing Institute of Technology, Beijing, 100081, China}
\thanks{L. Xie (elhxie@ntu.edu.sg) is with the School of Electrical and Electronic Engineering, Nanyang Technological University, Singapore 639798.}
\thanks{J. Chen (chenjie@bit.edu.cn) is with the School of Electronic and Information Engineering, Tongji University, Shanghai, 200082, China, and also with Key Laboratory of Intelligent Control and Decision of Complex Systems, School of Automation, Beijing Institute of Technology, Beijing, 100081, China.}
}

\markboth{IEEE Transactions on Automatic Control,~\today~\currenttime}
{\MakeLowercase{\textit{et al.}}: Distribute method}


\maketitle

\begin{abstract}
 This paper proposes a distributed stochastic projection-free algorithm for large-scale constrained finite-sum optimization whose constraint set is complicated such that the projection onto the constraint set can be expensive. The global cost function is allocated to multiple agents, each of which computes its local stochastic gradients and communicates with its neighbors to solve the global problem. Stochastic gradient methods enable low computational cost, while they are hard and slow to converge due to the variance caused by random sampling. To construct a convergent distributed stochastic projection-free algorithm, this paper incorporates a variance reduction technique and gradient tracking technique in the Frank-Wolfe update. We develop a sampling rule for the variance reduction technique to reduce the variance introduced by stochastic gradients.  Complete and rigorous proofs show that the proposed distributed projection-free algorithm converges with a sublinear convergence rate and enjoys superior complexity guarantees for both convex and non-convex objective functions. By comparative simulations, we demonstrate the convergence and computational efficiency of the proposed algorithm.
\end{abstract}

%
\IEEEpeerreviewmaketitle

\section{Introduction}
%
%
%
%
Large-scale finite-sum optimization with some special structures, such as low-rank or sparsity constraints, has wide applications in  various data analysis and machine learning tasks \cite{recomm_lowrank} \cite{tensor_de} \cite{pred_control}. Recently, distributed/decentralized finite-sum optimization over multi-agent networks has attracted much interest. 
On the one hand, finite-sum optimization problems deployed on network systems naturally demand decentralized computations, which reduce the computational burden in the center node, especially for large-scale models. On the other hand, it is often expensive to communicate or unallowable to transmit all information to a central processing node in network systems. Hence, the development of  efficient distributed algorithms  for solving the large-scale constrained finite-sum optimization is of great importance.
\par Much recent research effort has been devoted to the design and analysis of distributed optimization algorithms. Most distributed algorithms are built on the average consensus idea to find a solution of an optimization problem \cite{dis_timevary_direct,ZOU2020109060,Shi_extra,dis_fix_op}. To obtain a better convergence with non-diminishing step-sizes, gradient tracking technique is integrated in the design of distributed optimization algorithms \cite{geo_con,Qu_harn,Xu_disg}. If the optimization problems are constrained by convex sets, the existing distributed algorithms fall into two main categories. One category is primal projected gradient descent algorithms \cite{dis_gra_cons,constr_nons_res,proj_longcon}, which are based on the premise that the projection on constraint sets is simple. The second category of algorithms is primal-dual algorithms \cite{con_finite_sum,non_constrained,smooth_Const}, which usually involve constructing a dual optimal set containing the dual optimal variable. Furthermore, for large-scale constrained finite-sum optimization, where the calculation of exact local gradients is expensive, some distributed stochastic algorithms estimate the local gradients by sampling data and use projection operators \cite{Bianchi_prostog,dis_yu_zhan,sto_NNLS_yuan} or primal-dual frameworks \cite{stoc_pd_online,primal_CDC,bianchi_primal} to obtain feasible solutions.
\par Despite of these progress on distributed optimization, there are two main challenges in developing distributed stochastic algorithms for large-scale finite-sum optimization with complex constraints. One challenge is that computing  projections onto constraint sets can be expensive when the constraint set is highly complex. Because of this, the popular projected gradient descent (PGD) algorithms, in which a projection onto the constraint set is applied in each iteration of algorithms, are inefficient and can be even intractable \cite{Reddi_16_nonconvex}.
Another challenge is the variance on gradients of stochastic algorithms, which compute stochastic gradients by sampling data.
Although  stochastic algorithms significantly  reduce computation and storage costs compared with deterministic algorithms, distributed stochastic  algorithms are usually slow and hard to converge due to the variance caused by random sampling and distributed data.

\par To overcome the challenge caused by expensive projection operators of complex sets, distributed Frank-Wolfe algorithms are one of the promising solvers for large-scale constrained finite-sum optimization in recent years. Unlike projected gradient algorithms, Frank-Wolfe (FW) algorithms are projection-free and solve  linear minimization over constraint sets to make variables feasible \cite{fran_wolf_work,block_Fw}. FW algorithms are often significantly less costly than projected gradient algorithms, especially for the unit-norm constraints and nuclear norm constraints in machine learning community, over which linear minimization owns closed-form solutions \cite{hazan_Fw,pmlr-v28-jaggi13}. Although many centralized FW-type algorithms have been developed, distributed FW-type algorithms over multi-agent networks have not been well studied until recent years.
\cite{FW-TAC} developed a decentralized FW-type (DenFW) algorithm and provided the convergence analysis by viewing the decentralized algorithm as an inexact FW algorithm. For the special monotone and continuous DR-submodular maximization, \cite{pmlr-v89-xie19b} developed a distributed FW-type algorithm with a low communication complexity.  To improve the communication efficiency of distributed FW algorithms for a general constrained finite-sum optimization, \cite{Xian_Huang_Huang_2021} proposed a decentralized quantized Frank-Wolfe algorithm, which filled the gap of decentralized quantized constrained optimization. 

\par  Recently, there has been increasing interest \cite{vr-de,comm_de} in  dealing  with the challenges from the variance on stochastic gradients of centralized FW-type methods, which are much younger compared with various stochastic projected gradient methods \cite{MADEIRA2000427,YUAN2018196}.  
For the convex finite-sum optimization, \cite{pmlr-v48-hazana16}  developed stochastic FW algorithms with and without variance reduction techniques.
For the non-convex finite-sum optimization,  \cite{Reddi_16_nonconvex} proposed a stochastic Frank-Wolfe (SFW) method and a stochastic variance reduced variant of FW algorithm (SVFW). To achieve an $\epsilon$-solution, SFW requires $\mathcal{O}(\frac{1}{\epsilon^4})$ incremental first-order oracle (IFO) and SVFW improves the IFO of SFW to $\mathcal{O}(\frac{n^{2/3}}{\epsilon^2})$, where $n$ denotes the total number of  samples. Unlike SFW and SVFW that use unbiased gradient estimators, \cite{pmlr-v97-yurtsever19b} combined a stochastic path-integrated differential estimator technique (SPIDER) \cite{SPIDER-FANG} with the classical FW method to develop a new variant SPIDER-FW so that the IFO is improved to $\mathcal{O}(\frac{n^{1/2}}{\epsilon^2})$. Compared with SVFW, the improvement is significant for a large-scale dataset where $n$ is huge. However, the research of distributed stochastic projection-free algorithms with variance reduction is still absent.  
\par  Inspired by the existing centralized stochastic works \cite{Reddi_16_nonconvex,pmlr-v97-yurtsever19b} and distributed optimization algorithms \cite{Xu_disg,pmlr-v89-xie19b,FW-TAC}, we develop a distributed stochastic projection-free algorithm for large-scale constrained finite-sum optimization.
The contributions of this paper are summarized as follows.
\begin{itemize}
	\item[(1)] We propose a distributed stochastic projection-free algorithm, called DstoFW algorithm, by combining the Frank-Wolfe method and a variance reduction technique. The proposed DstoFW algorithm computes stochastic local gradients and communicates with neighbors only one round per iteration, resulting in less computational burden than the distributed deterministic works. 
	\item[(2)] We design a new periodically reducing sampling rule for the variance reduction step of DstoFW algorithm.  With this sampling rule, we establish that the proposed distributed stochastic algorithm has the same sublinear convergence rates as those of the deterministic algorithm \cite{FW-TAC}, i.e. an $\mathcal O(1/k)$  convergence rate for convex objective functions and an $\mathcal O(1/\sqrt{K})$  convergence rate for non-convex objective functions.  
 	\item[(3)] We provide the complexity analysis for the proposed distributed stochastic FW algorithm. We show that the proposed DstoFW has an incremental first-order oracle complexity of $\mathcal{O}(n^{3/4}\epsilon^{-1})$ for convex finite-sum optimization and an incremental first-order oracle complexity of  $\mathcal{O}(n^{2/3}\epsilon^{-2})$ for non-convex finite-sum optimization. 
\end{itemize}
 \par The remainder of the paper is organized as follows. The problem description and the proposed distributed DstoFW algorithm are given in Section \ref{solver_design}. The convergence results are given in Section \ref{conv_re}. The convergence and complexity analysis of the proposed method are provided in Section \ref{proof_sec}. The efficiency of the distributed algorithm is validated by comparative simulations in Section \ref{simulation} and the conclusion is made in Section \ref{conclusion}.
\subsection{Mathematical notations \& multi-agent graphs}
\par We denote $\mathbb{R}$ as the set of real numbers, $\mathbb{R}^+$ as the set of non-negative numbers, $\mathbb{Z}^+$ as the set of positive integers, $\mathbb{R}^n$ as the set of $n$-dimensional real column vectors, respectively. $\mathbf{1}$ denotes a column vector with all elements of $1$. All vectors in the paper are column vectors, unless otherwise stated. For a real vector $v$, $\left\|v\right\|$ is the Euclidean norm. The binary operator $\langle\cdot,\cdot \rangle$ denotes the inner product. The notation $[n]$ is the abbreviation of the set $\{1,\cdots,n\}$. For a set $S$, $|S|$ denotes the number of elements in the set $S$. The gradient of a differentiable function $f(x)$ is represented by $\nabla f(x)$. The notation ${\rm mod}(a,b)$, where $a,b \in \mathbb{Z}^+$, denotes the modulo operation. For a real number $c\in \mathbb{R}$, the operator $\lceil c \rceil$  returns the smallest integer that is greater than or equal to $c$.
\par The communication of multi-agent network system is modeled as an undirected graph $\mathcal{G}=(\mathcal{V},\mathcal{E},W)$, where $\mathcal{V}$ is a finite nonempty set of nodes, $\mathcal{E}\subset \mathcal{V}\times \mathcal{V}$ is
the corresponding set of edges. The adjacent matrix $W$ satisfies that $W_{ij}=W_{ji}>0$ if $\{i,j\}\in \mathcal{E}$ and the elements $W_{ij}= 0$, otherwise. In addition, the adjacent matrix $W$ is doubly stochastic, i.e., $W\mathbf{1}=W^\top\mathbf{1}=\mathbf{1}$. If the edge $\{i,j\}\in \mathcal{E}$ holds, agent $j$ is called a neighbor of agent $i$ and $\mathcal{N}_i$ is the set of neighbors of agent $i$.

\section{Problem description and algorithm design}\label{solver_design}

In this paper, we aim to solve the following large-scale constrained finite-sum optimization over a multi-agent network
\begin{align}\label{opti_p}
	\min_{x\in \Omega} F(x), \quad F(x)=\frac{1}{m}\sum_{i=1}^m f_i(x), \ f_i(x)= \frac{1}{n_i} \sum_{j=1}^{n_i} f_{i,j}(x),
\end{align}
where $x\in \Omega \subset \mathbb{R}^d$ is the unknown decision variable, $\Omega$ is a compact and convex constraint set, $f_i:\Omega \to \mathbb{R}$ is the local differentiable function of agent $i$, $m$ is the number of agents in the network and $n_i$ is the number of local data samples. In the multi-agent network, each agent $i$ only knows local information $f_{i}$ and communicates with its neighbors over the network $\mathcal{G}$ to solve \eqref{opti_p} cooperatively.

\par In problem \eqref{opti_p}, the volume of local data can be large such that the computation of full local gradients is time-consuming. In addition, the set $\Omega$ can be complex such that the calculation overhead for projection onto $\Omega$ can be heavy. To address these issues, we design a  distributed stochastic Frank-Wolfe algorithm (DstoFW), which is summarized in Algorithm \ref{fw_vr}. In the multi-agent network, each agent $i$ owns a local variable estimate $x_i^k$ at iteration $k$. Algorithm \ref{fw_vr} is composed of the following steps.
\par \textit{Average consensus step}: Each agent takes a weighted average of the variable estimates from its neighbors to approximate the average  $\bar{x}^k\triangleq \frac{1}{m}\sum_{i=1}^m x_i^k$. 
\par \textit{Frank-Wolfe step:} To avoid the complex projection operation, each agent takes a linear minimization over the constraint set $\Omega$ and a convex combination of the minimum $u_i^k$ and $\overline{x_i^k}$ to obtain a feasible variable estimate $x_i^{k+1}$ in the constraint set $\Omega$. The step-size $\gamma_k$ is $\frac{2}{k+1}$ for the convex optimization and $\frac{1}{\sqrt{k}}$ for the non-convex optimization.
\par\textit{ Variance reduction step:} To address the variance on gradient caused by random samples, each agent employs a stochastic path-integrated differential estimator technique to approximate the local gradient. The number of random samples $|S^k|$ satisfies the sampling rule 1 in Section \ref{conv_re}.
\par\textit{ Gradient tracking step:} To estimate the gradient of global function $F$, each agent uses the aggregated gradient $d_i^k$ from the last iteration and communicates with neighbors in \eqref{line_g} and \eqref{line_t}. In addition, variables are initialized as $d_i^1=v_i^1=g_i^1= \nabla f_i(x_i^1)$ for each agent $i$.
\begin{algorithm}
	\caption{Distributed stochastic FW algorithm (DstoFW) on agent $i$}
	\label{fw_vr}
	\begin{algorithmic}[1]
		\State \textbf{Input:} number of iterations $K$, initial condition $x_i^1\in \Omega$, $d_i^1=v_i^1=g_i^1= \nabla f_i(x_i^1)$.
		\For {$k=1,\cdots,K$}
		\State \textit{Average consensus step:} 
		    \begin{align}\label{line_2}
			\overline{x_i^k}=\sum_{j\in \mathcal{N}_i} W_{ij}x_j^k.
			\end{align}
		\State \textit{Frank-Wolfe step}: 
		\begin{align}
			u_i^{k}&={\rm argmin}_{u\in \Omega} \langle u, d_i^{k}\rangle \label{u_up}\\
			x_i^{k+1}&=(1-\gamma_k)\overline{x_i^k}+\gamma_k u_i^{k}\label{xup}
		\end{align}
		\State \textit{Variance reduction step:} 
		\If{$\mod(k+1,q)=0$}
		\begin{align}\label{mod_up}
		v_i^{k+1}=\frac{1}{n_i}\sum_{j=1}^{n_i} \nabla f_{i,j}({x_i^{k+1}})
		\end{align}
		\Else
		\State Select a sample set $S^{k}$ from $[n_i]$ to compute
		\begin{align}\label{samp_up}
			v_i^{k+1}=\frac{1}{|S^{k}|} \sum_{j\in S^{k}} [\nabla f_{i,j}({x_i^{k+1}})-\nabla f_{i,j}({x_i^k})]+v_i^k
		\end{align}
		\EndIf
		\State \textit{Gradient tracking step:} 
		\begin{align} g_i^{k+1}&=d_{i}^{k}+v_{i}^{k+1}-v_i^k\label{line_g}\\
		d_i^{k+1}&=\sum_{j\in \mathcal{N}_i} W_{ij}g_j^{k+1}\label{line_t}
		\end{align}
		\EndFor
	\end{algorithmic}
\end{algorithm}

\begin{remark}
	To our best knowledge, this is the first distributed stochastic projection-free algorithm for constrained optimization. Compared with the recent decentralized deterministic algorithm in \cite{pmlr-v89-xie19b,FW-TAC}, this proposed DstoFW introduces a variance-reduced stochastic technique to avoid the computational burden of full local gradients in large-scale problems. In addition, the proposed variance reduction step helps solve the slow convergence problem of stochastic algorithms caused by the random sampling, which enables the proposed algorithm to have the same convergence rates as deterministic ones. 
\end{remark}
\begin{remark}
	The sample sets in the variance reduction step satisfy a periodically reducing sampling rule, which is given in Section \ref{conv_re}. For convex (non-convex) optimization, we choose a suitable epoch $q=\lfloor n_i^{1/4} \rfloor$ ($q=\lfloor n_i^{1/3} \rfloor$) and the number of each sample set in the epoch satisfies the sampling rule 1. The designed sampling rule ensures the global gradient tracking of the proposed distributed algorithm after introducing the variance-reduced stochastic technique.
\end{remark}

\begin{remark}
	Using the gradient tracking and average consensus idea, this paper extends the existing centralized variance-reduced stochastic methods \cite{Reddi_16_nonconvex,pmlr-v97-yurtsever19b} to a distributed setting and reduces the communication congestion and computational burden of the central processing node. In addition, the proposed DstoFW needs only one communication round per iteration to exchange $x_i^k$ and $g_i^k$, whereas the decentralized algorithm \cite{FW-TAC} requires two communication rounds per iteration. 
\end{remark}
\begin{table*}[h]
	\centering
	\caption{Complexity comparisons of different algorithms}\label{complex_table}
	\begin{tabular}{|c|cc|cc|}
		\hline
		\multirow{2}{*}{Algorithms} & \multicolumn{2}{c|}{convex}      & \multicolumn{2}{c|}{non-convex}  \\ \cline{2-5}
		& \multicolumn{1}{c|}{IFO}  & LO   & \multicolumn{1}{c|}{IFO}  & LO   \\ \hline
		SFW\cite{pmlr-v48-hazana16,Reddi_16_nonconvex}                         & \multicolumn{1}{c|}{$\mathcal{O}(\epsilon^{-3})$} & $\mathcal{O}(\epsilon^{-1})$  & \multicolumn{1}{c|}{$\mathcal{O}(\epsilon^{-4})$} & $\mathcal{O}(\epsilon^{-2})$ \\ \hline
		SVFW \cite{pmlr-v48-hazana16,Reddi_16_nonconvex}                       & \multicolumn{1}{c|}{$\mathcal{O}(mn\ln(\epsilon^{-1})+\epsilon^{-2})$} & $\mathcal{O}(\epsilon^{-1})$ & \multicolumn{1}{c|}{$\mathcal{O}((mn)^{2/3}\epsilon^{-2})$} & $\mathcal{O}(\epsilon^{-2})$ \\ \hline
		SPIDER-FW \cite{pmlr-v97-yurtsever19b}                  & \multicolumn{1}{c|}{$\mathcal{O}(mn\ln(\epsilon^{-1})+\epsilon^{-2})$} & $\mathcal{O}(\epsilon^{-1})$ & \multicolumn{1}{c|}{$\mathcal{O}((mn)^{1/2}\epsilon^{-2})$} & $\mathcal{O}(\epsilon^{-2})$ \\
		\hline
		\textbf{DstoFW }                    & \multicolumn{1}{c|}{$\bm{\mathcal{O}(n^{3/4}\epsilon^{-1})}$} & $\bm{\mathcal{O}(\epsilon^{-1})}$ & \multicolumn{1}{c|}{$\bm{\mathcal{O}(n^{2/3}\epsilon^{-2})}$} & $\bm{\mathcal{O}(\epsilon^{-2})}$ \\ \hline
	\end{tabular}
\end{table*}
\section{Convergence results}\label{conv_re}
\par Now, we provide the convergence results of the proposed DstoFW algorithm. The proofs of all the convergence results will be postponed to next section. For the optimization problem \eqref{opti_p}, we make the following assumption.
\begin{assumption}\label{rho_assum}
	The constraint set $\Omega$ is convex and compact with diameter $\bar{\rho}$, i.e. $\bar{\rho}\triangleq \max_{x, \hat{x}\in \Omega} \|x-\hat{x}\|$.
\end{assumption}
{\textbf{Fact 1:}} Under Assumption \ref{rho_assum}, there exist constants $G, L, C\in \mathbb{R}^+$ such that the cost function $f_i$ is $G$-Lipschitz and $L$-smooth \cite{nest_convex}, and $\|\nabla f_i\|\leq C$.
\par In addition, for the multi-agent network, we assume that the adjacent matrix $W$ satisfies the following assumption.
\begin{assumption}\label{net_assum}
	The multi-agent network is connected and the adjacent matrix $W$ is doubly stochastic with non-negative elements. In addition, the second largest eigenvalue of $W$ is strictly less than one, i.e., $|\lambda_2(W)|<1$.
\end{assumption}
\begin{remark}
	The doubly stochastic adjacent matrix $W$ enables each agent to take a weighted average of its neighbors' information. The doubly stochastic $W$ satisfying $|\lambda_2(W)|<1$ exists for connected undirected and weight-balanced directed graphs, which are common in the distributed optimization \cite{con_finite_sum,smooth_Const,dis_con_wei}.
\end{remark}

\par The first theorem shows that the local variable estimate of each agent converges to the same average $\bar{x}^k$.
\begin{theorem}\label{con_theo}
	Suppose Assumptions \ref{rho_assum} and \ref{net_assum} hold. If $\gamma_k\!=\frac{A}{k^{\alpha}}$ with $\alpha \!\in (0,1]$ and $A\!>\!0$, then $\|x_i^k-\bar{x}^k\|\!=\mathcal{O}(1/k^{\alpha})$.
\end{theorem}
\par Note that the step-size of the proposed algorithm is diminishing. To provide some concise convergence rates of DstoFW, we develop the following periodically reducing sampling rule about the relationship between sample size and the step-size.
\vspace{0.2cm}\\
\textbf{Sampling Rule 1:}
Take any $\tilde n \in \mathbb{Z}^+$ and $k\in[ (\tilde n-1)q+1, \tilde n q-1)\cap \mathbb{Z}^+\!$. The sample set $S^k$ satisfies $\!\sqrt{|S^{\tilde{n}q-1}|}\!=q$ and
\begin{align}\label{st_ran}
	\sqrt{\frac{1}{|S^{k}|}}\gamma_{k}\leq \sqrt{\frac{1}{|S^{k+1}|}}\gamma_{k+1}.
\end{align} 
\begin{remark}
	Because the step-size $\gamma_k$ is diminishing, the sequence $\{|S^k|\}_{(\tilde{n}-1)q+1}^{\tilde{n}q-1}$ satisfying \eqref{st_ran} is monotonically decreasing.
\end{remark}
\begin{remark}
	We notice that sampling rule 1 holds during the period $[ (\tilde n-1)q+1, \tilde n q-1)$. We can choose a suitable $q$ so that sample size at the beginning of period is tractable. We also provide the choice of $q$ in Corollary \ref{nonconv_cor}  for the optimal complexity of DstoFW. 
   In practice, given any $|S^{\tilde{n}q-1}|=q^2\leq n_i$, the sample size at iterate $k \in \{ (\tilde n-1)q+1,\cdots, \tilde n q-2\}$ satisfying \eqref{st_ran} can be chosen as
  \begin{align}\label{sam_prac}
  	|S^k|=\Big\lceil\frac{\gamma_{k}^2}{\gamma_{\tilde{n}q-1}^2}|S^{\tilde{n}q-1}|\Big\rceil.
  	\end{align}
\end{remark}

\begin{remark}
     With the help of Sampling rule 1, the proposed stochastic DstoFW algorithm achieves the same convergence rates in terms of iteration $k$ as those of the recent deterministic algorithm \cite{FW-TAC}.
\end{remark}
\par If the objective functions in \eqref{opti_p} are convex, the convergence rate of DstoFW is shown as follows.
\begin{theorem}\label{fconvex_theo}
	Suppose all local objective functions $f_i$, $i=1,\cdots,m$, are convex. Set the step size as $\gamma_k=\frac{2}{k+1}$. Under Assumptions \ref{rho_assum}, \ref{net_assum} and Sampling rule 1,
	\begin{align}
		\mathbb{E}[F(\bar{x}^k)]-F(x^*)= \mathcal{O}(\frac{1}{k}),
	\end{align}
 where $x^*$ is an optimal solution to \eqref{opti_p}.
\end{theorem}

\par If the objective functions in \eqref{opti_p} are not guaranteed to be convex, we study the convergence of FW-gap \cite{pmlr-v97-yurtsever19b}, which is defined as
\begin{align}\label{fwgapg}
	g_k\triangleq \max_{x\in \Omega} \langle \nabla F(\bar{x}^k),\bar{x}^k-x \rangle.
\end{align}
 It follows from the definition that $g_k\geq 0$  for all $\bar{x}^k\in \Omega$ and $g_k=0$ when the iterate $\bar{x}_k$ is a stationary point to \eqref{opti_p}. The convergence result of DstoFW for possible non-convex objective functions is shown in the following theorem.
\begin{theorem}\label{fnonconvex_theo}
	  Set the step size as $\gamma_k=\frac{1}{k^{\alpha}}$ for some $\alpha \in (0,1]$. Under Assumptions \ref{rho_assum}, \ref{net_assum} and Sampling rule 1, for all $K\geq 2$,
	\begin{align}\label{eg_nonconvex}
		\min_{k\in [\frac{K}{2}+1,K]} \mathbb{E}[g_k]= \mathcal{O}(\frac{1}{K^{\min\{\alpha,1-\alpha\}}}).
	\end{align}
\end{theorem}
\par Next, we further provide the complexity measure of DstoFW for solving convex and non-convex optimization. Consider problem \eqref{opti_p} and the proposed DstoFW algorithm. A solution $\bar{x}^k$ is an \textit{$\epsilon$-accurate solution} if it satisfies $\mathbb{E}[F(\bar{x}^k)]-F(x^*)\leq \epsilon$ for convex objective functions, or it satisfies $\frac{1}{K/2} \sum_{k=K/2+1}^K \mathbb{E}[g_k]\leq \epsilon$ for non-convex objective functions.

\par To measure the complexity of DstoFW, we use the following black-box oracles \cite{Reddi_16_nonconvex,pmlr-v97-yurtsever19b}.
\begin{itemize}
	\item[$\cdot$] Incremental First-Order Oracle (IFO): For a function
	$F(\cdot) = \frac{1}{n}\sum_{i=1}^n f_i(\cdot)$, an IFO takes an index $i \in [n]$ and a point $x \in \mathbb{R}^d$, and returns the pair $(f_i(x), \nabla f_i(x))$. 
	\item[$\cdot$] Linear Optimization Oracle (LO): For a set $\Omega$, an LO takes a direction $d$ and returns ${\rm argmin}_{v\in \Omega} \langle v, d\rangle$.
\end{itemize}
 The IFO complexity (LO complexity)  is defined as the total number of IFO calls (LO calls) made by DstoFW to obtain an $\epsilon$-accurate solution. 

\begin{corollary}\label{nonconv_cor}
	Under Assumptions \ref{rho_assum}, \ref{net_assum} and Sampling rule 1. Let $n\triangleq \max\{n_1,\cdots,n_m\}$.
	\par (1) Convex case: Set the step-size $\gamma_k=\frac{2}{k+1}$ and $\sqrt{|S^{\tilde{n}_kq-1}|}=q=\lfloor n_i^{1/4}\rfloor$. The LO complexity of DstoFW is $\mathcal{O}(\epsilon^{-1})$ and the IFO complexity is
	$\mathcal{O}(\frac{n^{3/4}}{\epsilon})$
	\par (2) Non-convex case:
	 Set the step-size $\gamma_k=\frac{1}{\sqrt{k}}$ and $\sqrt{|S^{\tilde{n}_kq-1}|}=q=\lfloor n_i^{1/3}\rfloor$. The LO complexity of DstoFW is $\mathcal{O}(\epsilon^{-2})$ and the IFO complexity is $\mathcal{O}(\frac{n^{2/3}}{\epsilon^2})$.
\end{corollary}

\begin{remark}
	The choice of $q$  for convex (non-convex) optimization depends on both the number of local data and the sampling rule. For the convex case, since  $\sqrt{|S^{\tilde{n}q-1}|}=q$ in Sampling rule 1, and in view of \eqref{s_big2} and $|S^k|\leq n_i$ for all $k$, the choice of $q$ needs to satisfy $0<q\leq n_i^{1/4}$. In addition, the IFO complexity of the proposed algorithm is inversely proportional to $q$ by the theoretical proof, then, we obtain the optimal choice of $q$ to be $\lfloor n_i^{1/4}\rfloor$. The similar analysis holds for the non-convex case. 
\end{remark}
	\begin{remark}
	We compare the IFO and LO complexity of each agent in the proposed DstoFW algorithm with those of some existing centralized stochastic algorithms, which are summarized in Table \ref{complex_table}. Each agent in the proposed distributed DstoFW algorithm owns a better IFO complexity for convex finite-sum optimization, especially when  $\epsilon$ is small. For non-convex finite-sum optimization, the IFO complexity of  DstoFW is better than SFW and SVFW, but worse than SPDIER-FW when the number of local data $n$ is large. However, this drawback can be handled by using more agents in the distributed DstoFW algorithm to decrease the number of local data. 
	\end{remark}
\begin{table}
	\centering
	\caption{ Convergence rates of different algorithms}\label{rate_table}
	\begin{threeparttable}
	\begin{tabular}{|c|c|c|}
		\hline
		{Algorithms} & {convex} & \multicolumn{1}{l|}{non-convex} \\ \hline
		SVFW   \cite{pmlr-v48-hazana16,Reddi_16_nonconvex}                           & $\mathcal{O}(1/2^k)$                                                    & $\mathcal{O}(1/\sqrt{K})$                             \\ \hline
		SPIDER-FW\cite{pmlr-v97-yurtsever19b}                  & $\mathcal{O}(1/(2^k+t))$\tnote{1}                                                    & $\mathcal{O}(1/\sqrt{K})$                               \\ \hline
		\textbf{DstoFW}                           & $\bm{\mathcal{O}(1/k)}$                                                    & $\bm{\mathcal{O}(1/\sqrt{K})}$                             \\ \hline
	\end{tabular}
	\begin{tablenotes}
		\footnotesize
		\item[1] $t$ is the inner iterates of SPIDER-FW 
	\end{tablenotes}
\end{threeparttable}
\end{table}
	\begin{remark}
We further provide the comparison of convergence rates of the distributed DstoFW algorithm and some existing centralized stochastic Frank-Wolfe algorithms in Table \ref{rate_table}. To our best knowledge, there is no distributed stochastic Frank-Wolfe algorithm for finite-sum optimization except this work. It is observed that for non-convex finite-sum optimization, the proposed distributed DstoFW algorithm owns a same convergence rate as the existing centralized stochastic algorithms. Whereas, for convex finite-sum optimization, the distributed DstoFW algorithm owns a slower convergence rate than the centralized works, which is common in distributed algorithms.
\end{remark}
\begin{remark} 
{It is worth noting that for convex finite-sum optimization, recent centralized stochastic algorithms often requires more samples and IFO calls than that of DstoFW at each iteration $k$. For example, SVFW \cite{pmlr-v48-hazana16} requires $\mathcal{O}(4^k)$ IFO calls and SPIDER-FW \cite{pmlr-v97-yurtsever19b} requires $2^{2k-2}$ IFO calls, while DstoFW only requires $\mathcal{O}(1)$ IFO calls. Thanks to Sampling rule 1 and a suitable choice of $q$, the proposed distributed DstoFW owns a lower IFO complexity than the centralized stochastic algorithms despite the lower convergence rate for convex finite-sum optimization.}
	\end{remark}
\section{theoretical analysis}\label{proof_sec}
In this section, we present the theoretical proofs for the convergence and complexity performance of the proposed algorithm, stated in the theorems and corollary in section \ref{conv_re}. For the convenience of analysis, we define  some auxiliary quantities  $\bar{d}^k\triangleq\frac{1}{m}\sum_{i=1}^m d_i^k$, $\bar{g}^k\triangleq\frac{1}{m}\sum_{i=1}^m g_i^k$, $\bar{v}^k\triangleq\frac{1}{m}\sum_{i=1}^m v_i^k$, and $\overline{\nabla_k F}\triangleq \frac{1}{m}\sum_{i=1}^m \nabla f_i({x_i^k})$
\subsection{Consensus analysis}
\par In Algorithm \ref{fw_vr}, each agent takes a weighted average of the variable estimates from its neighbors in \eqref{line_2}. We firstly state one well-known fact for this average consensus update in the following proposition.
\begin{proposition}\cite[Fact 1]{FW-TAC}\label{barsumpro}
	Suppose Assumption \ref{net_assum} holds. Let $x_1,\cdots,x_m\in \mathbb{R}^d$ be variable estimates and $\bar{x}=\frac{1}{m}\sum_{i=1}^m x_i$ be their average. The update $\overline{x_i}=\sum_{j=1}^m W_{ij}x_j$ satisfies
	\begin{align}
		\sqrt{\sum_{i=1}^m \|\overline{x_i}-\bar{x}\|^2}\leq |\lambda_2(W)|\sqrt{\sum_{i=1}^m\|x_i-\bar{x}\|^2},
	\end{align}
where $\lambda_2(W)$ is the second largest eigenvalue of the adjacent matrix $W$.
\end{proposition}
\par For some $\alpha \in (0,1]$, we define $k_0(\alpha)$ as the smallest positive integer such that
\begin{align}\label{k0_def}
	\lambda_2(W)\leq \Big(\frac{k_0(\alpha)}{k_0(\alpha)+1}\Big)^{\alpha}\frac{1}{1+(k_0(\alpha))^{-\alpha}}.
\end{align}
Then, we show that for each agent $i$, the local approximation of average estimate $\overline{x_i^k}$ tracks the average estimate $\bar{x}^k$ with an approximation error. 
\begin{lemma}\label{maxbarxisub}
	Suppose Assumptions \ref{rho_assum} and  \ref{net_assum} hold. Set the step-size $\gamma_k=k^{-\alpha}$ for any $\alpha\in (0,1]$. We have
	\begin{align}\label{xsubbax}
		\max_{i\in[m]}\|\overline{x_i^k}-\bar{x}^k\|\leq C_p \gamma_k, \ \forall k\geq 1,
	\end{align}
where $C_p\triangleq (k_0(\alpha))^{\alpha}\sqrt{m}\bar{\rho}$.
\end{lemma}
\begin{proof}
	The proof is given in Section \ref{proof_lem}.
\end{proof}
It follows from \eqref{xsubbax} that the approximation error $C_p \gamma_k$ is diminishing with the iteration increasing. Now, with Lemma \ref{maxbarxisub} and the boundedness of $\Omega$, we are ready to prove the consensus performance of the proposed DstoFW algorithm.
\par \textbf{Proof of Theorem \ref{con_theo}:}
\begin{proof}
	Using the triangle inequality, we have
	\begin{align*}
		\|x_i^{k+1}\!-\!\bar{x}^{k+1}\|\leq \!\|x_i^{k+1}\!-\!\overline{x_i^k}\|\!+\!\|\overline{x_i^k}\!-\!\overline{x_i^{k+1}}\|\!+\!\|\overline{x_i^{k+1}}\!-\!\bar{x}^{k+1}\|.
	\end{align*}
	Then, we consider the three terms on the right-hand side of the above inequality, respectively. Because of \eqref{xup} and the boundedness of $\Omega$,
	\begin{align}\label{fir_ter}
		\|x_i^{k+1}-\overline{x_i^k}\|=&\|\gamma_k(u_i^k-\overline{x_i^k})\|
		\leq \gamma_k \bar{\rho}.
	\end{align}
	In addition, it follows from \eqref{xsubbax} and \eqref{fir_ter} that
	\begin{align}\label{barxkisub}
		\|\overline{x_{i}^{k+1}}\!-\!\overline{x_i^{k}}\|&=\Big\|\sum_{j=1}^m W_{ij}\big((x_j^{k+1}-\overline{x_j^{k}})+(\overline{x_j^{k}}-\overline{x_i^{k}})\big)\Big\|\notag\\
		&\leq \! \sum_{j=1}^m W_{ij}\big(\|x_j^{k+1}\!-\!\overline{x_j^{k}}\|\!+\!\|\overline{x_j^{k}}\!-\!\bar{x}^k\|\!+\!\|\overline{x_i^{k}}\!-\!\bar{x}^k\|\big)\notag\\
		&\leq \! \sum_{j=1}^m W_{ij}\big(\bar{\rho}\gamma_{k}+2 C_p\gamma_{k}\big)\notag\\
		&=(\bar{\rho}+2C_p)\gamma_{k},
	\end{align}
	where the last equality holds due to the doubly stochasticity of $W$. Since $\gamma_{k+1}\leq \gamma_{k}$ and in view of \eqref{xsubbax}, $\|\overline{x_i^{k+1}}-\bar{x}^{k+1}\|\leq C_p \gamma_{k}$.
	To sum up, we obtain 
	\begin{align}\label{xixsub}
		\|x_i^{k+1}-\bar{x}^{k+1}\|\leq&\bar{\rho}\gamma_k+(\bar{\rho}+2C_p)\gamma_{k}+C_p \gamma_{k}\notag\\
		=&(2\bar{\rho}+3C_p) \gamma_k\notag\\
		\leq & (2\bar{\rho}+3C_p)2^{\alpha} \gamma_{k+1},
	\end{align}
where the last inequality holds because
	\begin{align}\label{gammaksubk}
	\gamma_{k-1}=\gamma_k (\frac{k}{k-1})^{\alpha}=\gamma_k (1+\frac{1}{k-1})^{\alpha}
	\leq  \gamma_k 2^{\alpha}.
	\end{align}
This implies the desired result that $\|x_i^k-\bar{x}^k\|=\mathcal{O}(1/k^{\alpha})$.  
\end{proof}
\subsection{Convergence rate analysis}
Using the upper bound of $\|\overline{x_i^k}-\bar{x}^k\|$ in Lemma \ref{maxbarxisub}, we next discuss that for each agent $i$, the gradient estimate $d_i^k$ tracks the average $\bar{d}^k$ and the average $\bar{d}^k$ tracks the gradient $\overline{\nabla_k F}$ with an approximation error in the following proposition. 
\begin{proposition}\label{usubx_pro}
	Suppose Assumptions \ref{rho_assum} and \ref{net_assum} hold. Set the step size $\gamma_k=k^{-\alpha}$ and $C_p\triangleq (k_0(\alpha))^{\alpha}\sqrt{m}\bar{\rho}$, where $\alpha\in (0,1]$.
	Then, with Sampling rule 1, we have $\bar{d}^{k}=\bar{g}^k=\bar{v}^k$,
	\begin{align}\label{dusbf} \mathbb{E}[\|\bar{d}^k-\overline{\nabla_k F}\|]\leq L (5\bar{\rho}+7C_p) 2^{\alpha}\gamma_k,
	\end{align}
	\begin{align}\label{dibard}
		\mathbb{E}\Big[\sum_{i=1}^m \|d_i^k-\bar{d}^k\|^2\Big]\leq  C_g^2\gamma_k^2,
	\end{align}
	where  $C_g\triangleq k_0(\alpha)^{\alpha}\Big(2mL(5\bar{\rho}+7C_p)+\big(4m3^{k_0(\alpha)}C^2+12m3^{k_0(\alpha)}\mathbb{C}\big)^{1/2}\Big)$ and $\mathbb{C}\triangleq 2L^2 (5\bar{\rho}+7C_p)^2+ 2C^2$.
\end{proposition}
\begin{proof} The proof is given in Section \ref{proof_pro}.
\end{proof}

With the tracking results in Lemma \ref{maxbarxisub} and Proposition \ref{usubx_pro}, the following proposition indicates that the formula $\mathbb{E}[\langle u_i^k-\bar{x}^k, \nabla F(\bar{x}^k)\rangle]$ is bounded by the FW-gap.
\begin{proposition}
	Suppose Assumptions \ref{rho_assum} and \ref{net_assum} hold. Set the step size $\gamma_k=k^{-\alpha}$, where $\alpha\in (0,1]$. With Sampling rule 1, we have for all $u \in \Omega$,
	\begin{align}\label{uisubbarx}
		&\mathbb{E}[\langle u_i^k-\bar{x}^k, \nabla F(\bar{x}^k)\rangle]\notag \\
		\leq & \mathbb{E}[\langle u-\bar{x}^k,\nabla F(\bar{x}^k)\rangle] \!+\!2\bar{\rho}\Big(L (7\bar{\rho}+10C_p)2^{\alpha}  \gamma_k+C_g\gamma_k\Big).
	\end{align}
\end{proposition}
\begin{proof}
	\par Because of \eqref{u_up} and the boundedness of $\Omega$,
	\begin{align}\label{uisubxnabl}
		& \langle u_i^k-\bar{x}^k, \nabla F(\bar{x}^k)\rangle \notag\\
		\leq& \langle u_i^k-\bar{x}^k,d_i^k\rangle+\bar{\rho}\|\nabla F(\bar{x}^k)-d_i^k\|\notag\\
		\leq & \langle u-\bar{x}^k,d_i^k\rangle+\bar{\rho}\|\nabla F(\bar{x}^k)-d_i^k\|\notag\\
		= & \langle u-\bar{x}^k,\nabla F(\bar{x}^k)+d_i^k-\nabla F(\bar{x}^k)\rangle+\bar{\rho}\|\nabla F(\bar{x}^k)-d_i^k\|\notag\\
		\leq & \langle u-\bar{x}^k,\nabla F(\bar{x}^k)\rangle +2\bar{\rho}\|\nabla F(\bar{x}^k)-d_i^k\|,\quad \forall u \in \Omega.
	\end{align}
Recall that $\overline{\nabla_k F}=\frac{1}{m}\sum_{i=1}^m \nabla f_i({x_i^k})$ and $\nabla F(\bar{x}^k)=\frac{1}{m}\sum_{i=1}^m \nabla f_i(\bar{x}^k)$. Then, the norm $\|\nabla F(\bar{x}^k)-d_i^k\|$ in \eqref{uisubxnabl} satisfies
\begin{align}\label{dsubnabF}
	&\|d_i^k-\nabla F(\bar{x}^k)\|\notag\\
	\leq & \|d_i^k-\bar{d}^k\|+\|\bar{d}^k-\overline{\nabla_k F}\|+\|\overline{\nabla_k F}-\nabla F(\bar{x}^k)\|\notag\\
	\leq & \|d_i^k-\bar{d}^k\|+\|\bar{d}^k-\overline{\nabla_k F}\|\!+\!\frac{1}{m}\sum_{i=1}^m \|\nabla f_i({x_i^k})-\nabla f_i(\bar{x}^k)\|\notag\\
	\leq & \|d_i^k-\bar{d}^k\|+\|\bar{d}^k-\overline{\nabla_k F}\|+\frac{L}{m}\sum_{i=1}^m \|{x_i^k}-\bar{x}^k\|.
\end{align}
In addition, because of $\mathbb{E}(\sqrt{X})\leq \sqrt{\mathbb{E}(X)}$ and \eqref{dibard}, we have
 \begin{align}\label{normdsub}
 \mathbb{E}[\|d_i^k-\bar{d}^k\|]\leq &\mathbb{E}\sqrt{\sum_{i=1}^m \|d_i^k-\bar{d}^k\|^2}\notag\\
 \leq &\sqrt{\mathbb{E}\Big[\sum_{i=1}^m \|d_i^k-\bar{d}^k\|^2\Big]}\leq C_g\gamma_k.
 \end{align}
Substituting \eqref{xixsub} to \eqref{dsubnabF} and taking expectation of both sides,
\begin{align}\label{exp_d}
	&\mathbb{E}[\|d_i^k-\nabla F(\bar{x}^k)\|]\notag\\
	\leq \quad & \mathbb{E}[\|d_i^k-\bar{d}^k\|+\|\bar{d}^k-\overline{\nabla_k F}\|]+L (2\bar{\rho}+3C_p)2^{\alpha}\gamma_{k}\notag\\
	\overset{\eqref{dusbf},\eqref{normdsub}}{\leq} & C_g\gamma_k+ L (5\bar{\rho}+7C_p)2^{\alpha}  \gamma_k+L (2\bar{\rho}+3C_p)2^{\alpha}\gamma_k\notag\\
	=\quad &L (7\bar{\rho}+10C_p)2^{\alpha}  \gamma_k+C_g\gamma_k.
\end{align}
Then, taking expectation of \eqref{uisubxnabl} and using \eqref{exp_d} yields $ \forall u \in \Omega$,
\begin{align*}
	&\mathbb{E}[ \langle u_i^k-\bar{x}^k, \nabla F(\bar{x}^k)\rangle]\\
	\leq & \mathbb{E}[\langle u-\bar{x}^k,\nabla F(\bar{x}^k)\rangle] +2\bar{\rho}\Big(L (7\bar{\rho}+10C_p)2^{\alpha}  \gamma_k+C_g\gamma_k\Big).
\end{align*}
\end{proof}
\par Now, we are ready to prove that the proposed DstoFW converges for both convex and non-convex optimization. Before that, we introduce one lemma from \cite[Lemma2]{momen_lemm}.
\begin{lemma}\label{phi_des_lem}
	Let $\phi(t)$ be a sequence of real numbers satisfying
	\begin{align*}
		\phi(t)=\Big(1-\frac{A}{(t+t_0)^{r_1}}\Big)\phi(t-1)+\frac{B}{(t+t_0)^{r_2}},
	\end{align*}
for some $r_1\in [0,1]$ such that $r_1\leq r_2\leq 2r_1$, $A>1$ and $B\geq 0$. Then, $\phi(t)$ converges to zero at the following rate
$$\phi(t)\leq \frac{Q}{(t+t_0+1)^{r_2-r_1}},$$
where $Q=\max\{\phi(0)(t_0+1)^{r_2-r_1},\frac{B}{A-1}\}$.
\end{lemma}
\par \textbf{Proof of Theorem \ref{fconvex_theo}:}
\begin{proof}
	Taking the average of both sides of  \eqref{xup}, we obtain
	\begin{align}\label{barxsub}
		\bar{x}^{k+1}-\bar{x}^k=\gamma_k(\frac{1}{m}\sum_{i=1}^m u_i^{k}-\bar{x}^k).
	\end{align}
Let $\triangle F^k\triangleq F(\bar{x}^k)-F(x^*)$, where $x^*$ is an optimal solution to \eqref{opti_p}. It follows from the $L$-smoothness of $F$ and the boundedness of $\Omega$ that
\begin{align}\label{nabl}
	\triangle F^{k+1}\leq \triangle F^k+ \frac{\gamma_k}{m}\sum_{i=1}^m \langle \nabla F(\bar{x}^k),  u_i^k-\bar{x}^k\rangle+\frac{L}{2}\gamma_k^2\bar{\rho}^2.
\end{align}
Taking expectation of \eqref{nabl} and using  \eqref{uisubbarx} yields
\begin{align*}
	\mathbb{E}[\triangle F^{k+1}]\leq & \mathbb{E}[\triangle F^k]+ \gamma_k \mathbb{E}[\langle\bar{u}^k \!-\! \bar{x}^k,\nabla F(\bar{x}^k) \rangle]+\frac{L}{2}\gamma_k^2\bar{\rho}^2\\
	&+2 \gamma_k \bar{\rho}\Big(L (7\bar{\rho}+10C_p)2^{\alpha}  \gamma_k+C_g\gamma_k\Big),
\end{align*}
where $\bar{u}^k \in \Omega$ is the minimizer of the following linear optimization
\begin{align}\label{barudef}
	\bar{u}^k\in {\rm argmin}_{u\in \Omega} \langle u, \nabla F(\bar{x}^k)\rangle.
\end{align}
Substituting $\gamma_k={\frac{2}{k+1}}$ into the above inequality leads to
\begin{align}\label{kbetw}
	\mathbb{E}[\triangle F^{k+1}]\leq & \mathbb{E}[\triangle F^k]+ \frac{2}{k+1} \mathbb{E}[\langle\bar{u}^k-\bar{x}^k,\nabla F(\bar{x}^k) \rangle]\notag\\
	&\!+\!\frac{2L}{(k +1)^2}\bar{\rho}^2\!+\!8  \bar{\rho} \frac{2L (7\bar{\rho}\!+\!10C_p) +C_g}{(k+1)^2}.
\end{align}
Because of the convexity of $F$ and the definition of $\bar{u}^k$ in \eqref{barudef},
\begin{align*}
	\langle\bar{u}^k-\bar{x}^k,\nabla F(\bar{x}^k) \rangle\leq \langle x^*-\bar{x}^k,\nabla F(\bar{x}^k) \rangle\leq -\triangle F^k.
\end{align*}
Plugging the above inequality into \eqref{kbetw}, we obtain
\begin{align*}
	\mathbb{E}[\triangle F^{k+1}]\leq& (1-\frac{2}{k+1})\mathbb{E}[\triangle F^k]+\frac{2L}{(k+1)^2}\bar{\rho}^2\\
	&+8  \bar{\rho} \frac{2L (7\bar{\rho}+10C_p) +C_g}{(k+1)^2}.
\end{align*}
Then, using Lemma \ref{phi_des_lem} with $r_1=1,\ r_2=2, A=2, t_0=1$ and $B=2L\bar{\rho}^2+8\bar{\rho}(2L (7\bar{\rho}+10C_p) +C_g)$,
\begin{align}\label{convex_re}
	\mathbb{E}[\triangle F^k]\leq \frac{Q_c}{k+2}=\mathcal{O}(\frac{1}{k}),
\end{align}
where $Q_c=\max \{2(F(\bar{x}^1)-F(x^*)),B\}.$
\end{proof}
\par \textbf{Proof of Theorem \ref{fnonconvex_theo}:}
\begin{proof}
	Because $F$ is $L$-smooth,
	\begin{align}\label{Fsmooth}
		F(\bar{x}^{k+1})\leq& F(\bar{x}^k)+\langle \nabla F(\bar{x}^k),\bar{x}^{k+1}-\bar{x}^k\rangle \!+\!\frac{L}{2}\|\bar{x}^{k+1}\!-\!\bar{x}^k\|^2\notag\\
		\leq &F(\bar{x}^k)+\langle \nabla F(\bar{x}^k),\bar{x}^{k+1}-\bar{x}^k\rangle+\frac{L}{2}\gamma_k^2 \bar{\rho}^2,
	\end{align}
where the last inequality holds because $	\|\bar{x}^{k+1}-\bar{x}^k\|=\gamma_k\|\frac{1}{m}\sum_{i=1}^m (u_i^{k}-\overline{x_i^k})\|\leq \gamma_k\bar{\rho}$.
	We recall the definition of FW gap
	\begin{align}
		g_k\triangleq \max_{x\in \Omega} \langle \nabla F(\bar{x}^k),\bar{x}^k-x\rangle=\langle \nabla F(\bar{x}^k),\bar{x}^k-\bar{u}^k\rangle,
	\end{align}
	where $\bar{u}^k$ is defined in \eqref{barudef}.
	Taking expectation of both sides of \eqref{Fsmooth} and using \eqref{uisubbarx}, \eqref{barxsub}, 
	\begin{align}\label{Fgap}
		\mathbb{E}[F(\bar{x}^{k+1})]\leq& \mathbb{E}[F(\bar{x}^k)]-\gamma_k \mathbb{E}[\langle \nabla F(\bar{x}^k),\bar{x}^k-\bar{u}^k\rangle]\notag\\
		&+\!2 \gamma_k \bar{\rho}\Big(L (7\bar{\rho}\!+\!10C_p)2^{\alpha}  \gamma_k\!+\!C_g\gamma_k\Big)\!+\!\frac{L}{2}\gamma_k^2 \bar{\rho}^2\notag\\
		=& \mathbb{E}[F(\bar{x}^k)]-\gamma_k \mathbb{E}[g_k]+\frac{L}{2}\gamma_k^2 \bar{\rho}^2\notag\\
		&+2 \gamma_k \bar{\rho}\Big(L (7\bar{\rho}+10C_p)2^{\alpha}  \gamma_k+C_g\gamma_k\Big).
	\end{align}
	 Without loss of generality, we assume that $K$ is an even integer in the following analysis.
	Summing the two sides of \eqref{Fgap} from $k=\frac{K}{2}+1$ to $k=K$ yields
	\begin{align}\label{non_ineq}
		&\sum_{k=K/2+1}^{K}\gamma_k \mathbb{E}[g_k]\notag\\
		\leq& \mathbb{E}[F(\bar{x}_{K/2+1})]-\mathbb{E}[F(\bar{x}_{K+1})]\notag\\
		&+\!\sum_{k=K/2+1}^{K} \!\Big(2 \gamma_k^2 \bar{\rho}\big(L (7\bar{\rho}\!+\!10C_p) 2^{\alpha} +C_g \big)\!+\frac{L}{2}\gamma_k^2 \bar{\rho}^2\Big)\!\notag\\
		\leq & G\bar{\rho} + C_0\sum_{k=K/2+1}^K \frac{1}{k^{2\alpha}},
	\end{align}
where $C_0\triangleq 2\bar{\rho}\big(L (7\bar{\rho}+10C_p) 2^{\alpha} +C_g\big)+L\bar{\rho}^2/2$, $\gamma_k=\frac{1}{k^{\alpha}}$ and $G$-Lipschitz of $F$ are used in the last inquality.
\par The left-hand side of \eqref{non_ineq} satisfies
\begin{align}\label{lowb}
	\sum_{k=K/2+1}^{K}\gamma_k\mathbb{E}[g_k]\geq \min_{k\in[K/2+1,K]} \mathbb{E}[g_k] \sum_{k=K/2+1}^K \gamma_k.
\end{align}
In addition, by the integral test \cite{calculus} and $\gamma_k=\frac{1}{k^{\alpha}}$, for $\alpha \in (0,1]$,
\begin{align}\label{sumgamlim}
	\sum_{k=K/2+1}^K \gamma_k \geq& \int_{k=K/2+1}^{K+1}\frac{1}{k^{\alpha}}{\rm d}k\notag\\
	=&\frac{k^{1-\alpha}}{1-\alpha}\Big|_{K/2+1}^{K+1}\notag\\
	=& \frac{(K+1)^{1-\alpha}-(\frac{K+2}{2})^{1-\alpha}}{1-\alpha}\notag\\
	=&\frac{(K+1)^{1-\alpha}\big(1-(\frac{K+2}{2(K+1)})^{1-\alpha}\big)}{1-\alpha}\notag\\
	\geq &\frac{K^{1-\alpha}}{1-\alpha}\big(1-(\frac{2}{3})^{1-\alpha}\big),
\end{align}
where we use the fact that $(\frac{K+2}{2(K+1)})^{1-\alpha}\leq (\frac{2}{3})^{1-\alpha}$ for $K\geq 2$ in the last inequality.
\par When $\alpha \geq 0.5$, by the integral test,
\begin{align}\label{sumlim}
	\sum_{k=K/2+1}^K \frac{1}{k^{2\alpha}}\leq& \sum_{k=K/2+1}^K \frac{1}{k} \notag\\
	\leq & \int_{K/2}^{K} k^{-1}{\rm d}k\notag\\
	= &\ln k|_{K/2}^{K} \notag\\
	= & \ln 2.
\end{align}
Substituting \eqref{lowb}, \eqref{sumgamlim} and \eqref{sumlim} to \eqref{non_ineq} yields
	\begin{align*}
		\min_{k\in [\frac{K}{2}+1,K]} \mathbb{E}[g_k]\leq \frac{1}{K^{1-\alpha}}\frac{1-\alpha}{1-(2/3)^{1-\alpha}}(G \bar{\rho}+C_0\ln 2).
	\end{align*}
\par When $0<\alpha<0.5$,
\begin{align}
	\sum_{k=K/2+1}^K \frac{1}{k^{2\alpha}}\leq \int_{K/2}^K \frac{1}{k^{2\alpha}} {\rm d}k= \frac{2^{1-2\alpha}-1}{1-2\alpha}(\frac{K}{2})^{1-2\alpha}.
\end{align}
Then, the right-hand side of \eqref{non_ineq} is upper bounded by
\begin{align}\label{sumlim2}
	G\bar{\rho} + C_0\sum_{k=K/2+1}^K \frac{1}{k^{2\alpha}}\leq \Big(G\bar{\rho} + C_0 \frac{1-(1/2)^{1-2\alpha}}{1-2\alpha}\Big)K^{1-2\alpha}.
\end{align}
Substituting \eqref{lowb}, \eqref{sumgamlim} and \eqref{sumlim2} to \eqref{non_ineq} yields
	\begin{align*}
		\min_{k\in [\frac{K}{2}\!+\!1,K]}\! \mathbb{E}[g_k]\leq \frac{1}{K^{\alpha}}\frac{1\!-\!\alpha}{1\!-\!(2/3)^{1\!-\alpha}}\Big(\!G \bar{\rho}\!+\!\frac{C_0(1\!-\!(1/2)^{1\!-\!2\alpha})}{1-2\alpha}\Big).
	\end{align*}
\par Combining the two cases $0.5\leq \alpha\leq 1$ and $0<\alpha<0.5$, we get the desired result in \eqref{eg_nonconvex}.
\end{proof}
{
\par \textbf{Proof of Corollary \ref{nonconv_cor}:}
\begin{proof}
	(1) Convex optimization case: It follows from \eqref{convex_re} that, to get an $\epsilon$-solution, the number of iteration $K$ satisfies
	\begin{align}
		\mathbb{E}[\triangle F^K]\leq \frac{Q_c}{K+2}\leq \epsilon,
	\end{align}
	where $Q_c=\max \{2(F(\bar{x}^1)-F(x^*)),B\}$ and $B=2L\bar{\rho}^2+8\bar{\rho}(2L (7\bar{\rho}+10C_p) +C_g)$.
	It is sufficient to choose $K= \lfloor\frac{Q_c}{\epsilon}\rfloor$. Since LO is calculated once per  iteration, the LO complexity of DstoFW is $\mathcal{O}(\frac{1}{\epsilon})$.
	\par Then, we consider the IFO complexity of DstoFW. 
	Let $\tilde{n}_k$ be a positive integer such that $(\tilde{n}_k-1)q+1\leq k \leq \tilde{n}_kq-1$.
	Making use of \eqref{sam_prac} and $\gamma_k=\frac{2}{k+1}$, the sample size $|S^{(\tilde{n}_k-1)q+1}|$, which is the largest number in sequence $\{|S^t|\}_{(\tilde{n}_k-1)q+1}^{\tilde{n}_kq-1}$, satisfies 
	\begin{equation}\label{s_big}
	\left\{
	\begin{aligned}
		|S^1|&=\! \Big\lceil \frac{q^2}{2} |S^{q-1}|\Big\rceil, \ \tilde{n}_k=1,\\
		|S^{(\tilde{n}_k-1)q+1}|&=\! \Big\lceil(1+\frac{q-2}{(\tilde{n}_k-1)q+2})^2 |S^{\tilde{n}_kq-1}|\Big\rceil, \ \tilde{n}_k\geq 2.
	\end{aligned}
	\right.
	\end{equation}
 Using the fact that $1+\frac{q-2}{(\tilde{n}_k-1)q+2}<2$, we obtain
	\begin{equation}\label{s_big2}
	\left\{
	\begin{aligned}
		|S^1|&\leq q^2 |S^{q-1}|, \ \tilde{n}_k=1,\\
		|S^{(\tilde{n}_k-1)q+1}|&\leq 4|S^{\tilde{n}_kq-1}|, \ \tilde{n}_k\geq 2.
	\end{aligned}
	\right.
	\end{equation}
	In addition, for every $q$ iterations, there is a full local gradient computation, whose IFO is $n_i$.
	\par To guarantee that each element of $\{|S^t|\}_{(\tilde{n}_k-1)q+1}^{\tilde{n}_kq-1}$ in \eqref{s_big2} is smaller than $n_i$ ($n_i> 10$), we set $\sqrt{|S^{\tilde{n}_kq-1}|}=q= \lfloor n_i^{1/4}\rfloor$ for all $\tilde{n}_k\in \mathbb{Z}^+$. Recall that $n=\max\{n_1,\cdots,n_m\}$. Then, the number of IFO calls (denoted by $\#$(IFO)) of each agent $i$ in DstoFW algorithm satisfies
	\begin{align*}
	\# ({\rm IFO}) &\leq n_i+q^3|S^{q-1}|+\big(n_i+q4| S^{\tilde{n}_kq-1}|\big)(K/q-1)\\
	&\leq n_i+n_i^{5/4}+(n_i+4n_i^{3/4})Kn_i^{-1/4}\\
 &=\mathcal{O}(\frac{n^{3/4}}{\epsilon}).
	\end{align*}
	\par(2)  Non-convex optimization case:
Using \eqref{non_ineq} and $\gamma_k=1/{\sqrt{k}}$, we show that
\begin{align*}
	&\frac{1}{\sqrt{K}}\sum_{k=K/2+1}^K \mathbb{E}[g_k] \leq \sum_{k=K/2+1}^K \frac{1}{\sqrt{k}}\mathbb{E}[g_k]\\
	&\leq G\bar{\rho}+\frac{L\bar{\rho}^2+4\bar{\rho}(L(5\bar{\rho}+7C_p) 2^{\alpha}+C_g)}{2}\ln 2,
\end{align*}
where the last inequality holds due to \eqref{sumlim}.
Multiplying the both sides by $\frac{2}{\sqrt{K}}$ yields
\begin{align*}
	\frac{1}{K/2} \!\sum_{k=K/2+1}^K\! \mathbb{E}[g_k]\! \leq\! \frac{2G\bar{\rho}}{\sqrt{K}}\!+\!\frac{L\bar{\rho}^2\!+\!4\bar{\rho}(L(5\bar{\rho}\!+\!7C_p) 2^{\alpha} \!+\!C_g)}{\sqrt{K}}\!\ln 2.
\end{align*}
To get an $\epsilon$-solution, the number of iteration $K$ satisfies
$$ \frac{2G\bar{\rho}}{\sqrt{K}}+\frac{L\bar{\rho}^2+4\bar{\rho}(L(5\bar{\rho}+7C_p) 2^{\alpha}+C_g)}{\sqrt{K}}\ln 2\leq \epsilon,$$
such that $K= \mathcal{O}(\frac{1}{\epsilon^2})$. Since LO is computed once per iteration, the LO complexity of DstoFW is $\mathcal{O}(\frac{1}{\epsilon^2})$. The analysis of IFO is similar to that of the above convex case and the only difference is $\gamma_k=1/\sqrt{k}$. 
 Set $\sqrt{|S^{\tilde{n}_kq-1}|}=q=\lfloor n_i^{1/3}\rfloor$ for all $\tilde{n}_k\in \mathbb{Z}^+$. Then, the IFO of DstoFW for non-convex optimization is  $\# ({\rm IFO})\leq n_i+q^2|S^{q-1}|+(n_i+q2|S^{\tilde{n}_kq-1}|)(K/q-1)=\mathcal{O}(\frac{n^{2/3}}{\epsilon^2})$.
\end{proof}}
\section{Simulation}\label{simulation}
In this section, we demonstrate the proposed algorithm in solving the binary classification problem, which is a finite-sum optimization with the general form \eqref{opti_p},
where $f_i$ may be convex \cite{log_simu} or non-convex \cite{Huang2019}. For comparison, we apply the decentralized DenFW algorithm in \cite{FW-TAC}, the centralized SPIDER-FW algorithm (CenFW) in \cite{pmlr-v97-yurtsever19b} and the proposed distributed DstoFW algorithm to solve the  problem. The loss function $F$ and FW-gap $g_k$ over iterations $k$ of different algorithms will be compared in the simulation\footnote{The simulation codes are provided at https://github.com/managerjiang/VR-FW}. 
\par The distributed algorithms are applied over a same ten-agent undirected connected network ($m=10$), which are shown in Fig. \ref{topo_fig}, and each agent only knows the local information $f_i$. In simulation, the constraint set is set as an $l_1$ norm ball constraint $\Omega=\{x|\|x\|_1\leq R\}$. Then, $u_i^k$ in DstoFW admits a closed form solution 
\begin{align*}
	u_i^k={\rm argmin}_{u\in\Omega}\langle u,d_i^k\rangle=&R\cdot [0,\cdots,0,-{\rm sgn}[d_{i}^k]_j,0,\cdots,0]^\top \notag\\ & {\rm with}\ j={\rm argmax}_l\big|[d_{i}^k]_l\big|,
\end{align*}
where the notation $[d]_l$ denotes the $l$th element in vector $d$. We take $R=20$ of the constraint set in the simulation test.
 We take three publicly available real datasets, which are summarized in Table \ref{real_data}, to test the simulation result. 
\begin{table}[!htbp]
	\caption{Real datasets for binary classification}\label{real_data}
	\centering
	\begin{tabular}{c|c|c|c}
		\hline
		datasets & \#samples & \#features & \#classes  \\
		\hline
		$a9a$  & 32561& 123 &2\\
		\hline
		w8a & 64700&300&2\\
		\hline
		$covtype.binary$&581012&54&2\\
		\hline
	\end{tabular}
\end{table}

\begin{figure}
	\centering
	\includegraphics[width=7cm]{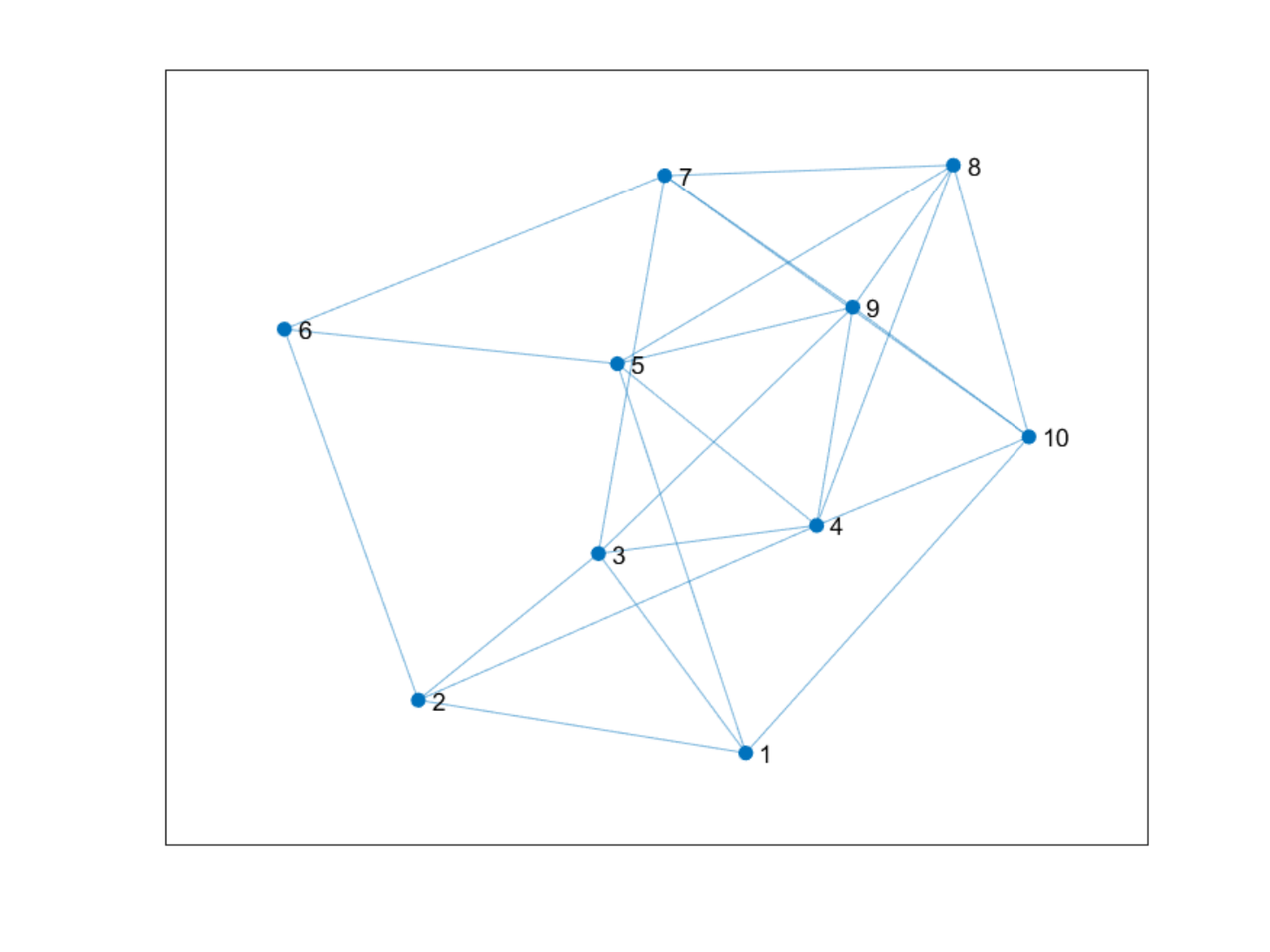}
	\caption{Ten-agent connected network}
	\label{topo_fig}	
\end{figure}
\begin{figure*}
	\centering
	\subfigure[a9a dataset]{
		\includegraphics[width=5.5cm]{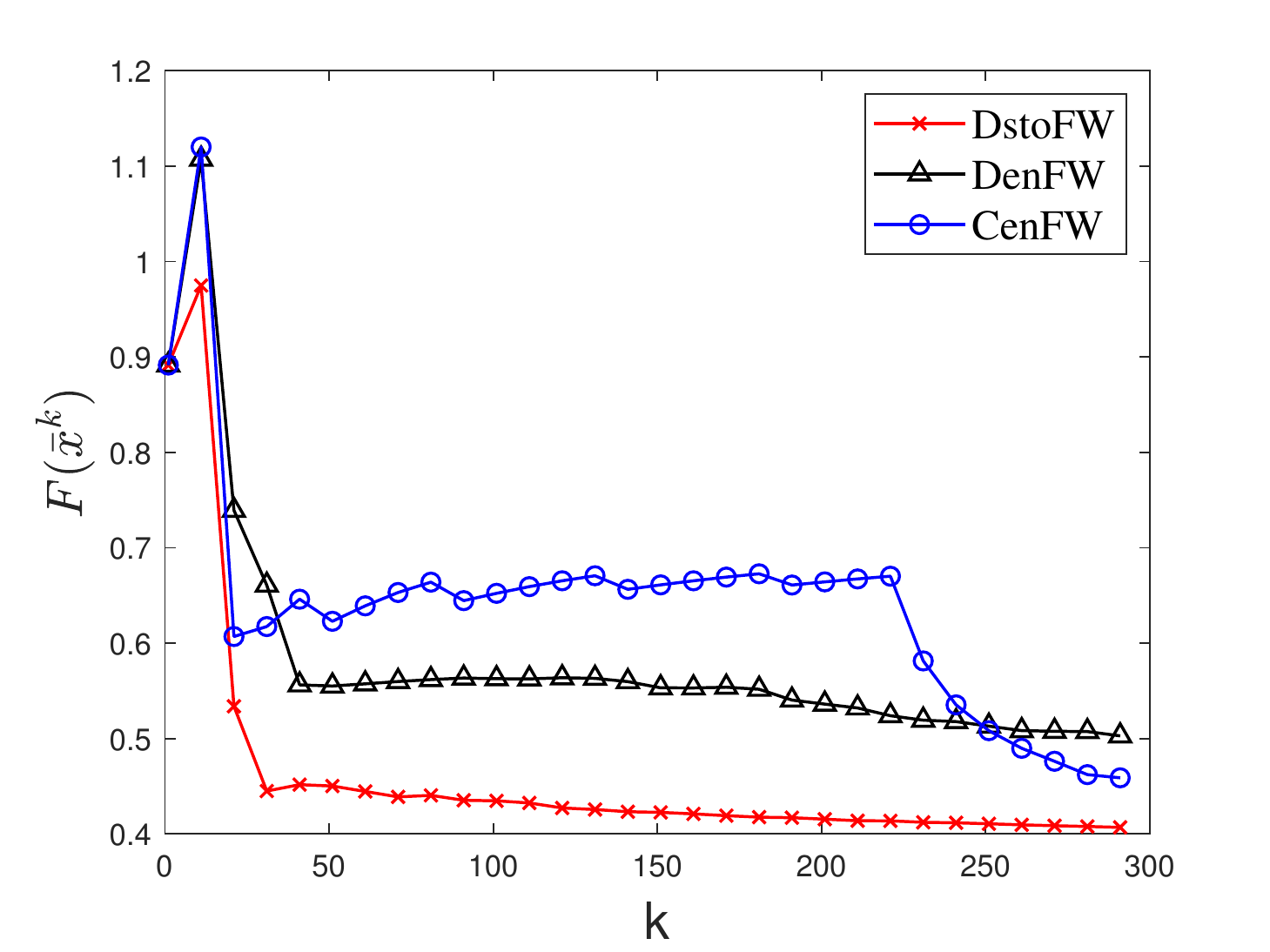}
	}
	\subfigure[w8a dataset]{
		\includegraphics[width=5.5cm]{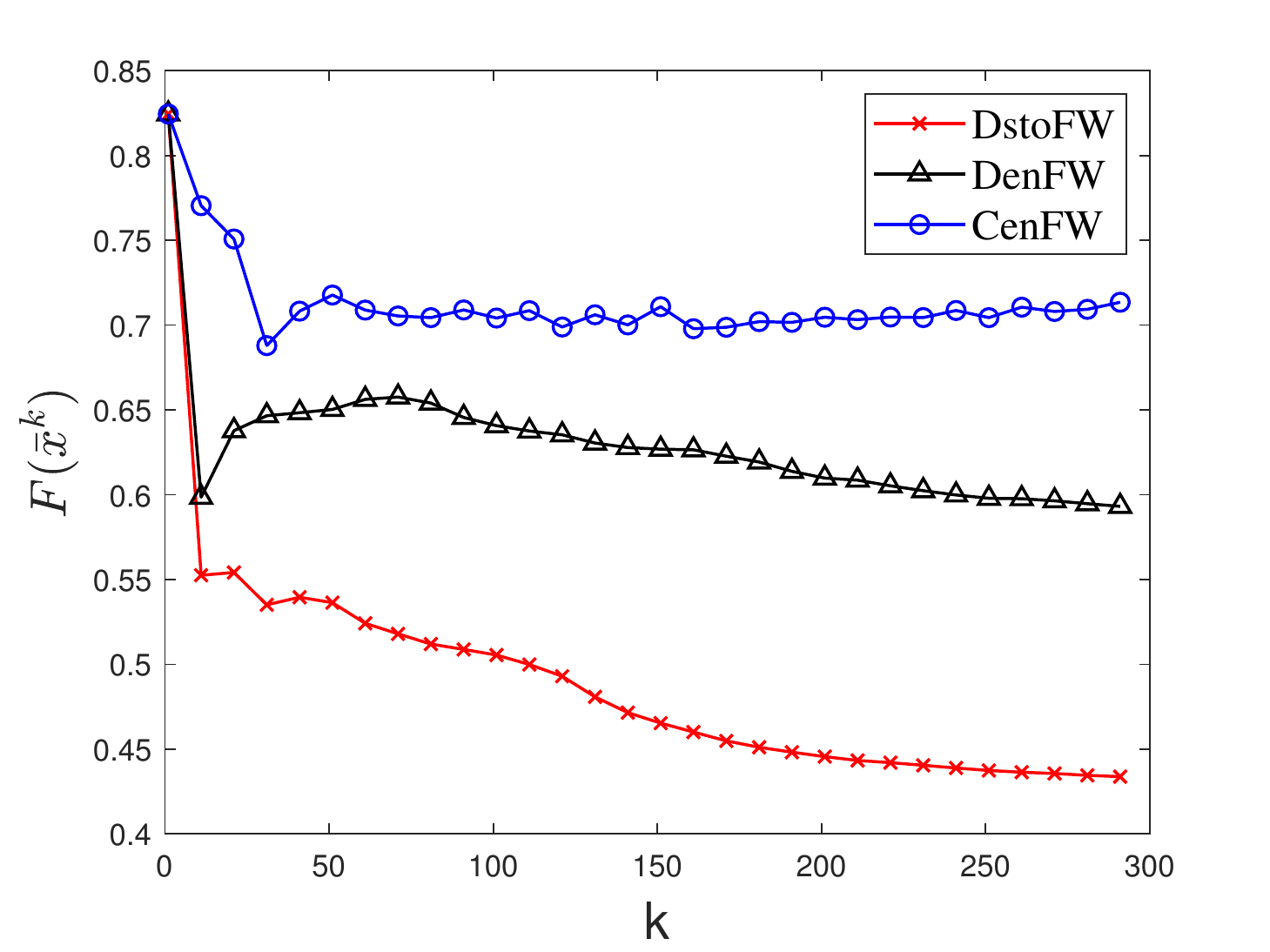}	
	}
	\subfigure[covtype dataset]{
		\includegraphics[width=5.5cm]{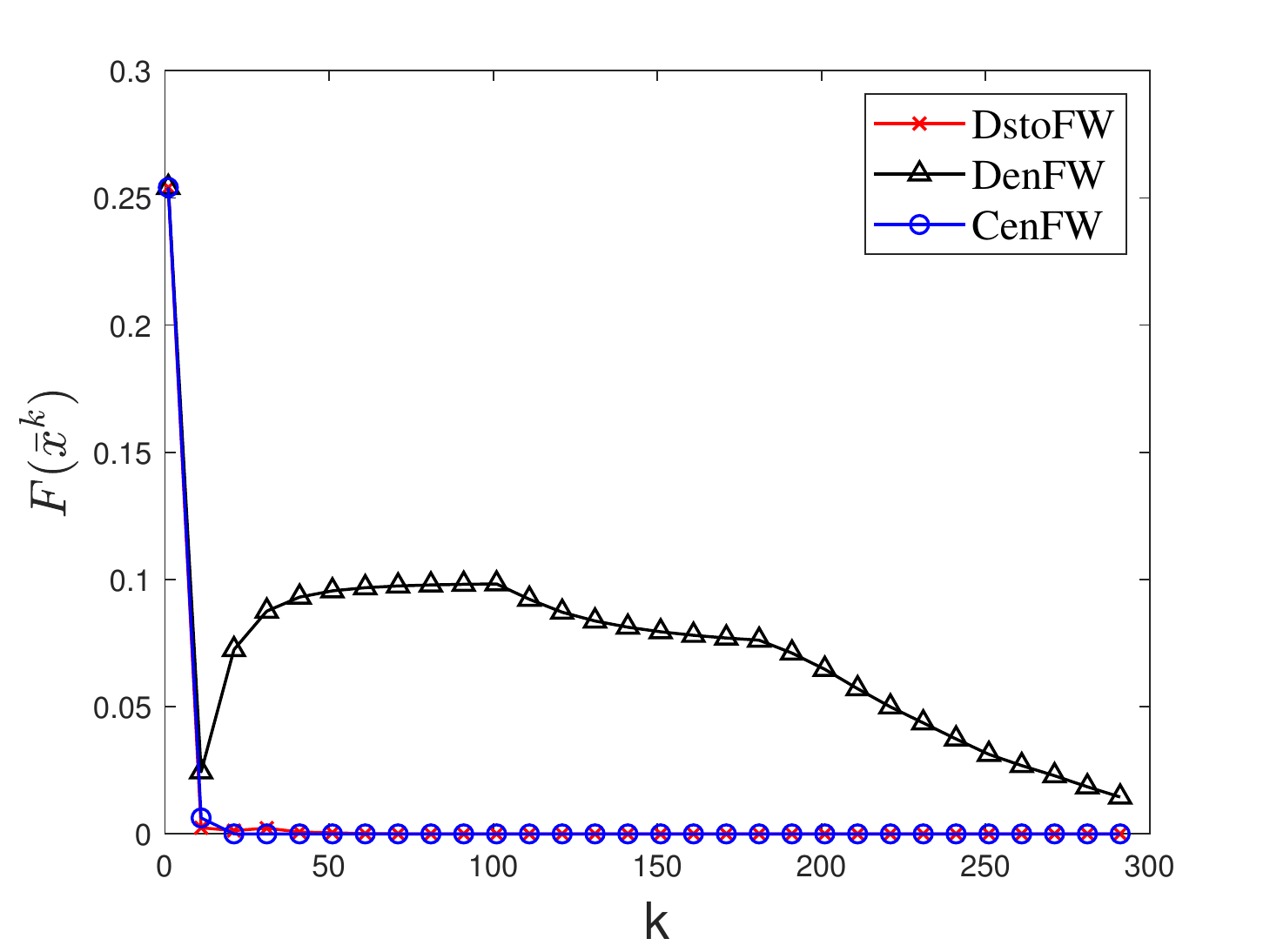}
	}
	\caption{ Different algorithms for convex optimization}
	\label{fwgap_fig}
\end{figure*}
 \begin{figure*}
	\centering
	\subfigure{
		\includegraphics[width=5.5cm]{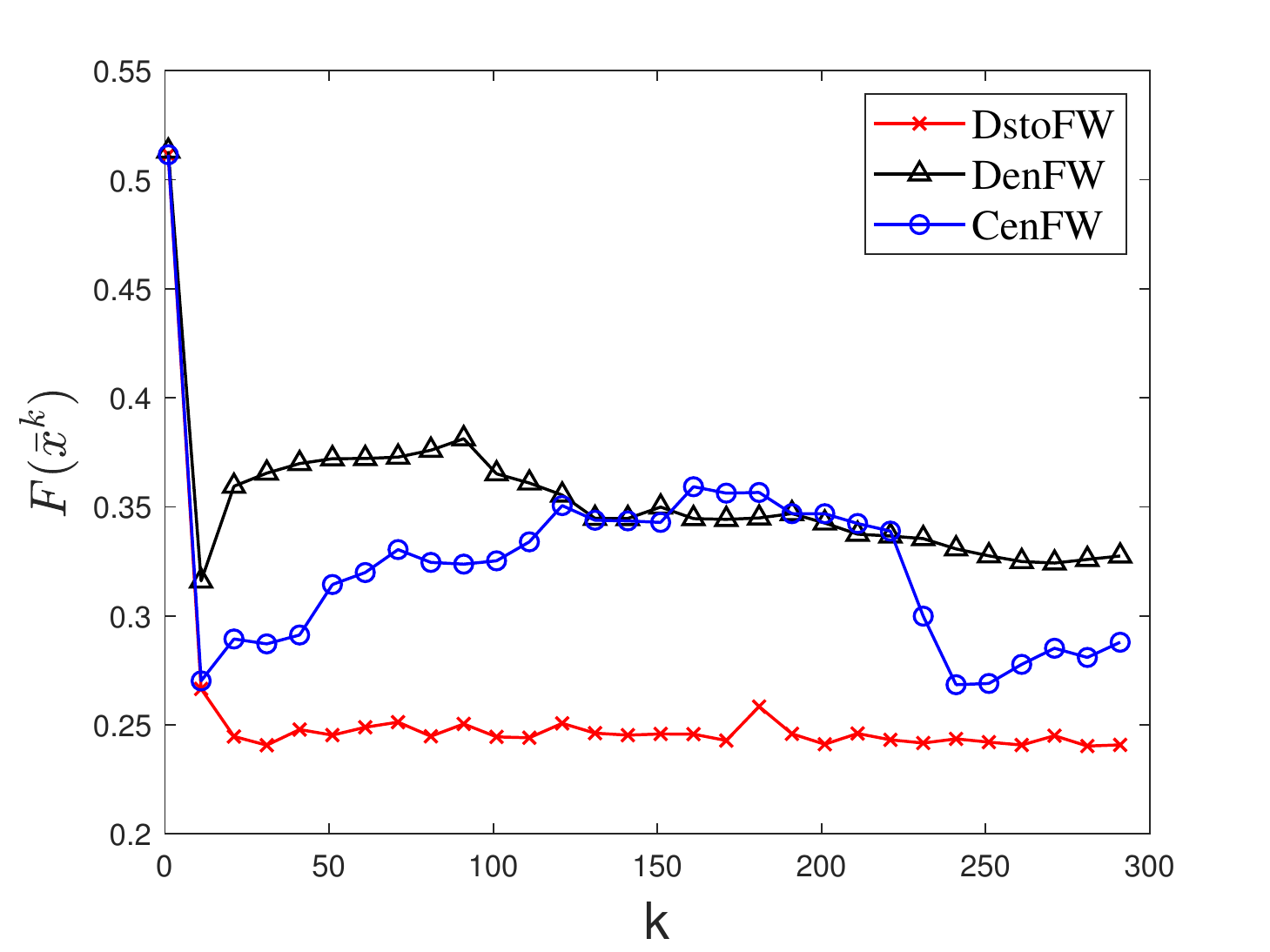}
	}
	\subfigure{
		\includegraphics[width=5.5cm]{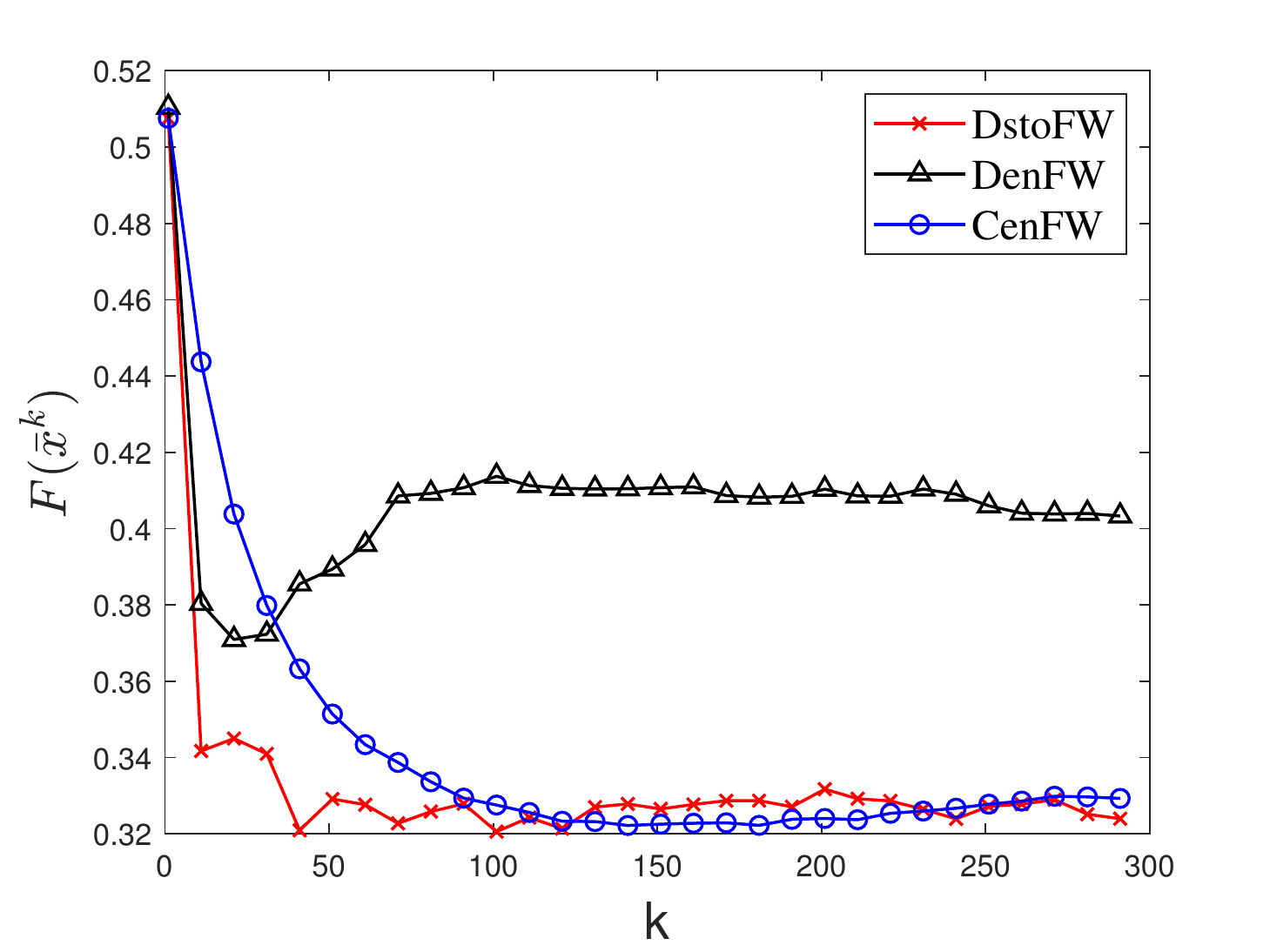}
	}
	\subfigure{
		\includegraphics[width=5.5cm]{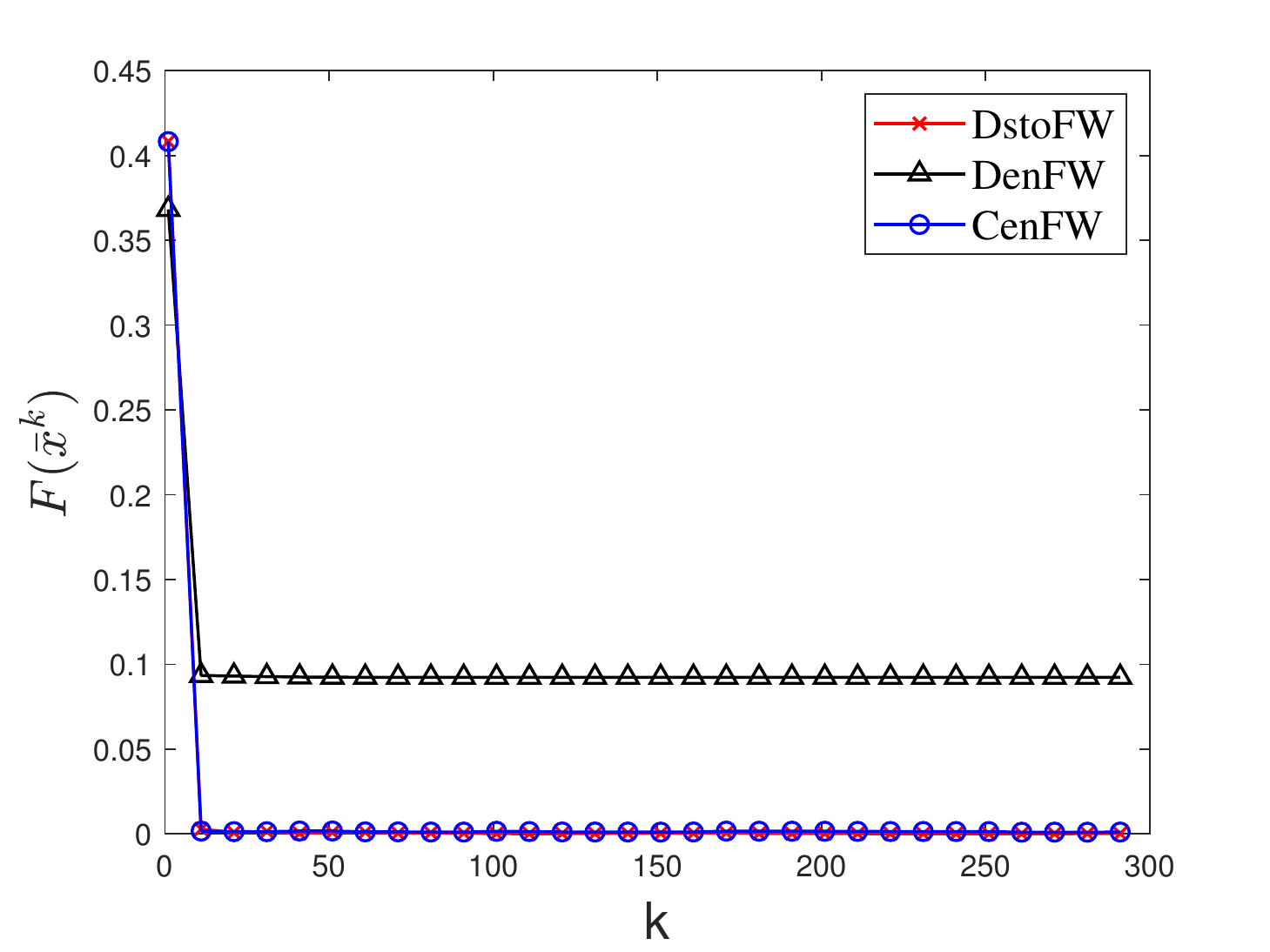}
	}
	\setcounter{subfigure}{0}
	\subfigure[a9a dataset]{
		\includegraphics[width=5.5cm]{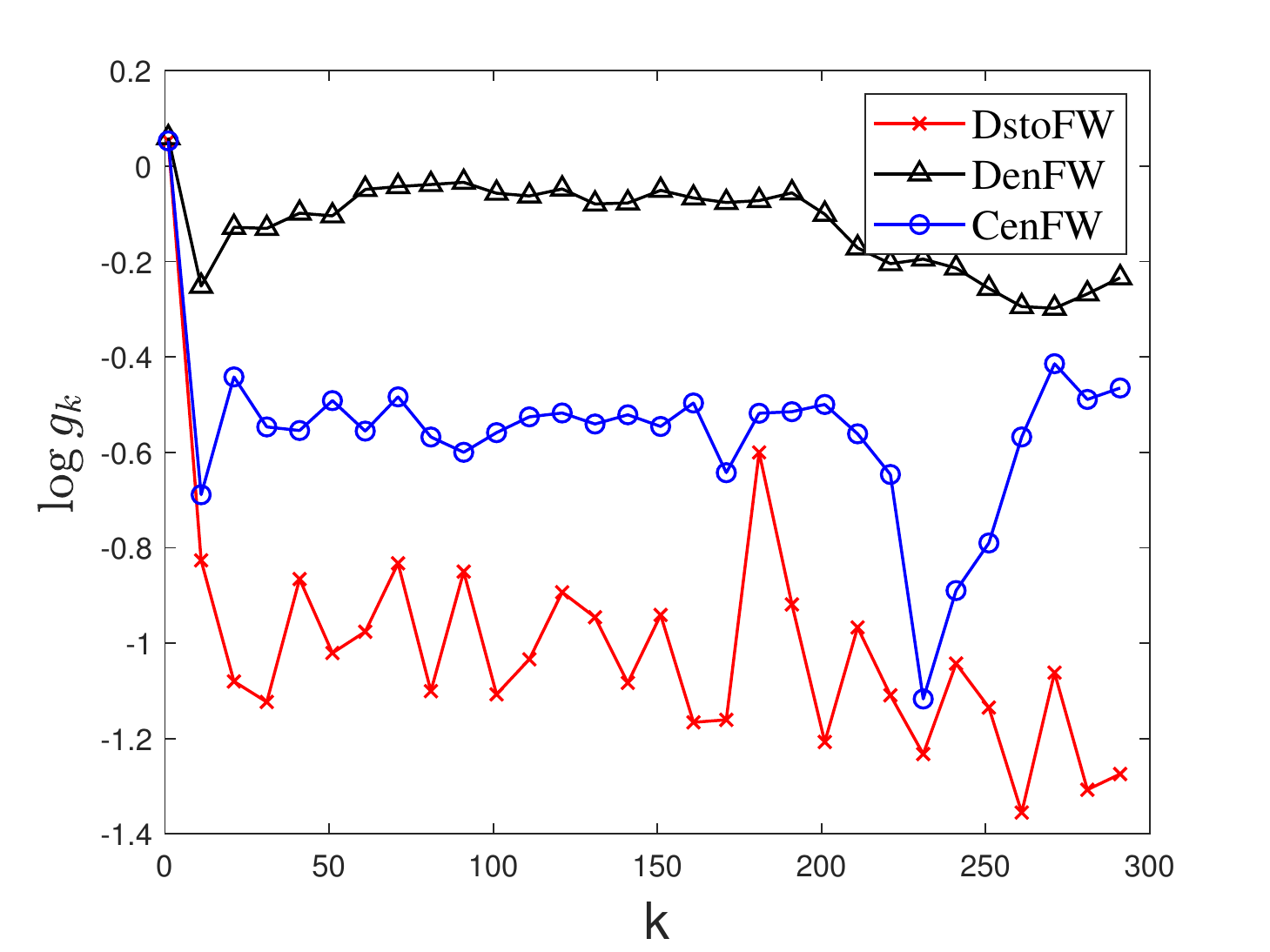}
		\label{a9a_fwgap_noncon}	
	}
	\subfigure[w8a dataset]{
		\includegraphics[width=5.5cm]{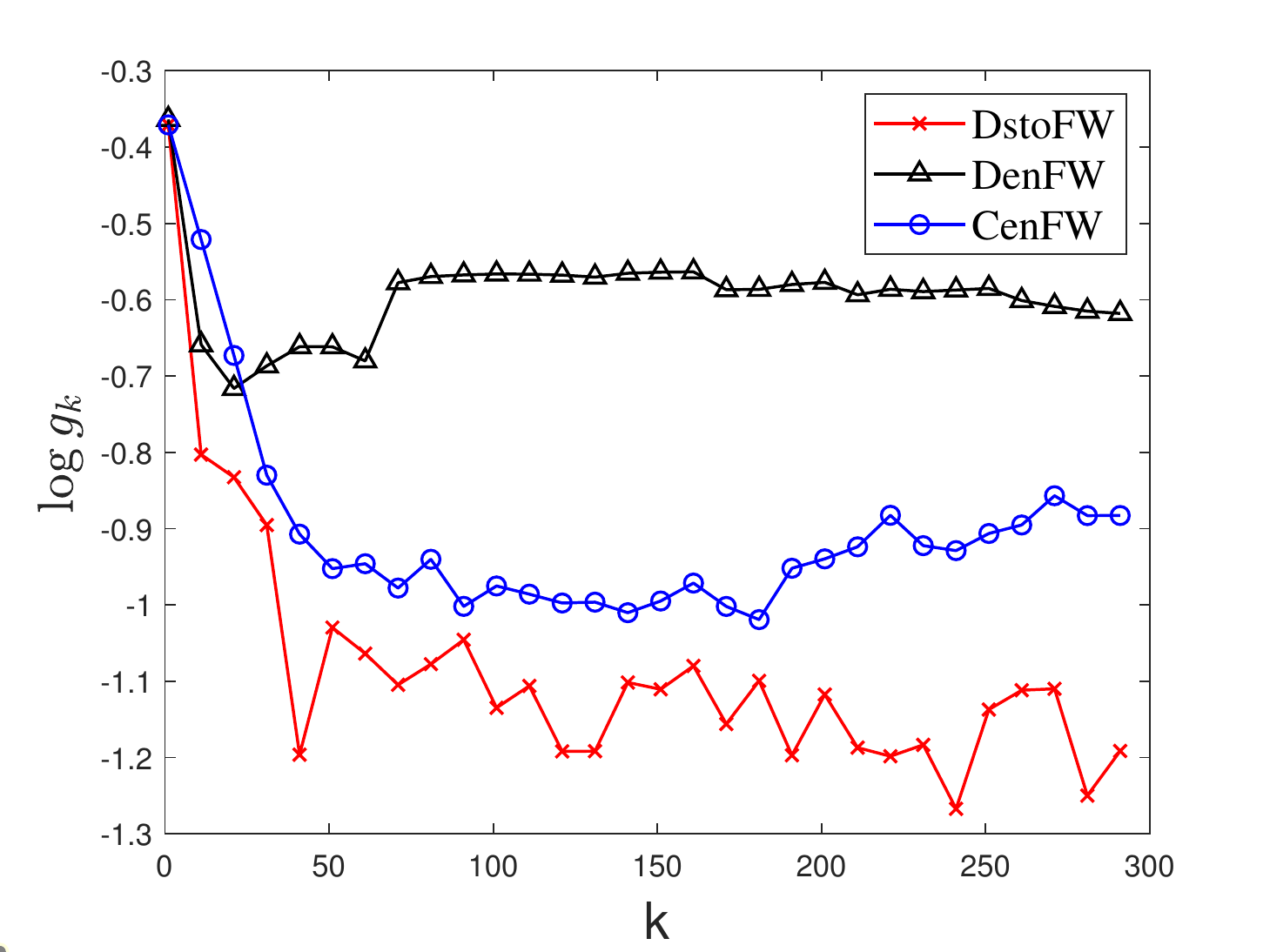}
		\label{w8a_fwgap_noncon}	
	}
	\subfigure[covtype dataset ]{
		\includegraphics[width=5.5cm]{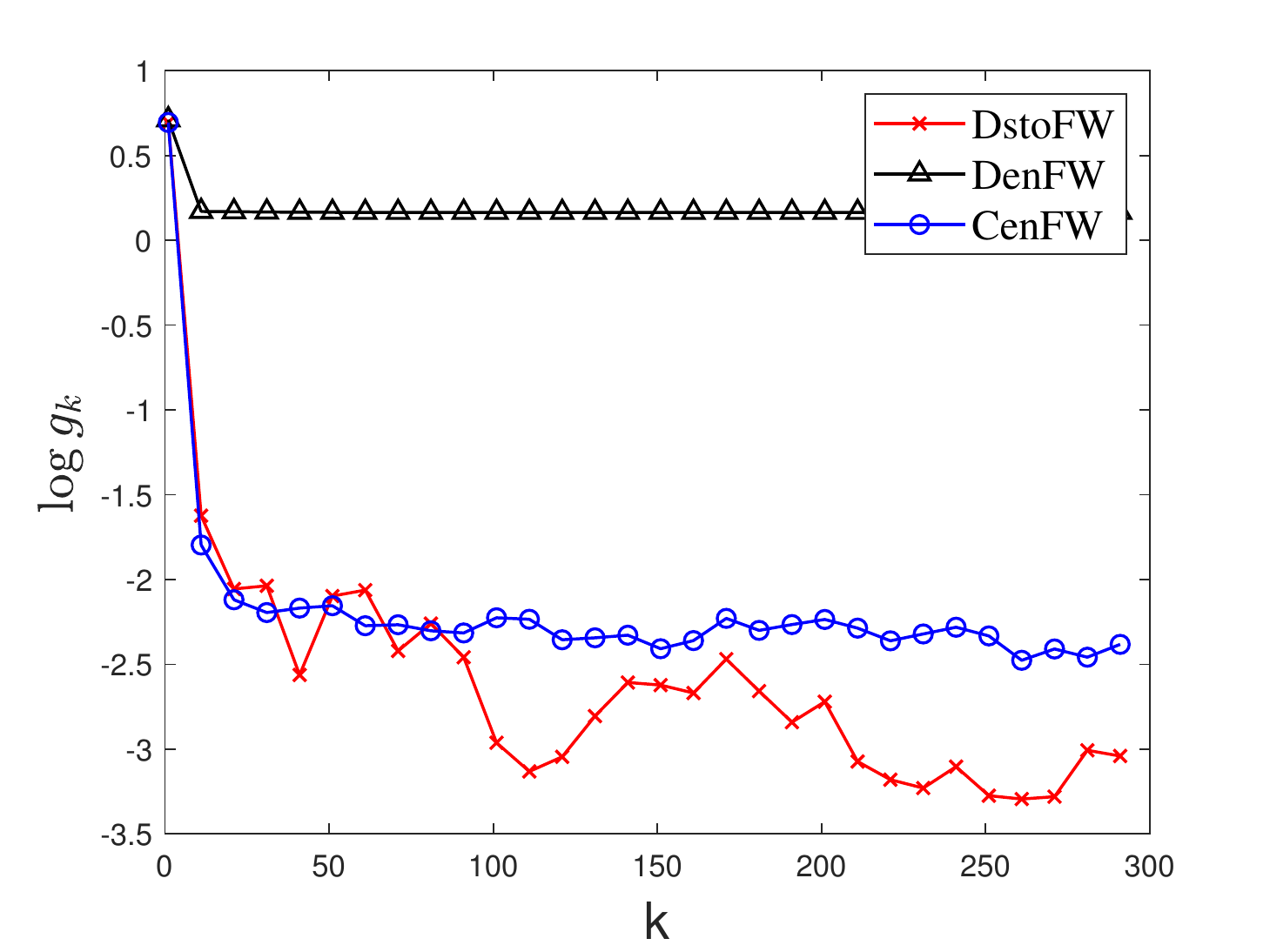}
		\label{cov_fwgap_noncon}	
	}
	\caption{ Different algorithms for non-convex optimization}
	\label{fwgap_noncon_fig}
\end{figure*}
\par (1) Convex objective function:  Using the convex logistic regression learning model in \cite{log_simu}, the convex local objective function $f_i$ in \eqref{opti_p} is
\begin{align}\label{convex_subp}
	f_i(x)=\frac{1}{n_i}\sum_{j=1}^{n_i} \ln(1+\exp(-l_{i,j}\langle a_{i,j},x \rangle)),
\end{align}
where $a_{i,j}\in \mathbb{R}^d$ is the feature vector of the $j$th local sample of agent $i$, $l_{i,j} \in \{-1,1\}$ is the classification value of the $j$th local sample of agent $i$ and $\{a_{i,j},l_{i,j}\}_{j=1}^{n_i}$ denotes the set of local training samples of agent $i$.
%

 In our simulation, for convenience, each agent in the network owns the same number of local sample data, i.e. $n_i=n_j$ for $i,j\in \{1,\cdots, 10\}$. Note that CenFW is a centralized algorithm. There is only one processing node ($m=1$) to deal with all samples of the dataset so that $\bar{x}^k=x^k$. Let $t\triangleq k//q$, where the notation $//$ denotes integer division operation. 
 According to \cite{FW-TAC}, \cite{pmlr-v97-yurtsever19b} and the proposed algorithm, the step-sizes of DenFW, CenFW and DstoFW are $2/(k+1)$, $2/(2^t+k+1)$ and  $2/(k+1)$, respectively.
  The simulation results of DenFW, CenFW and DstoFW algorithms for different datasets are shown in Fig. \ref{fwgap_fig}. It shows that for convex optimization, the DstoFW algorithm converges faster than the distributed DenFW and the centralized CenFW algorithms over different datasets.  

\par (2) Non-convex objective function:  Using the non-convex logistic regression learning model in \cite{Huang2019}, the non-convex local objective function $f_i$ in \eqref{opti_p} is
\begin{align}
	f_i(x)&=\frac{1}{n_i} \sum_{j=1}^{n_i} \frac{1}{1+\exp(l_{i,j}\langle a_{i,j},x\rangle)},
\end{align}
where $a_{i,j}$ and $l_{i,j}$ are same as those in \eqref{convex_subp}. The step-sizes of DenFW, CenFW and DstoFW are $1/\sqrt{k}$,  $1/\sqrt{K}$ and $1/\sqrt{k}$, respectively. The simulation results of DenFW, CenFW and DstoFW algorithms are shown in Fig. \ref{fwgap_noncon_fig}, where $g_k$ is the FW-gap defined in \eqref{fwgapg}. It is seen that for non-convex optimization, the proposed distributed stochastic DstoFW algorithm owns a better convergence performance than DenFW and CenFW algorithms. The trajectories of $g_k$ converging to zero implies that the generated variables of all algorithms converge to stationary points of the non-convex optimization. 
\begin{table}[]
	\centering
	\caption{Execution times of distributed algorithms}
	\label{exec_tab}
	\begin{tabular}{|c|c|c|c|}
		\hline
		\multicolumn{1}{|l|}{datasets} & algorithms & convex(s) & non-convex(s) \\ \hline
		& DenFW                                                    & 55.467    & 116.776       \\ \cline{2-4}
		\multirow{-2}{*}{a9a}          & DstoFW                                                   & 10.263    & 31.191        \\ \hline
		& DenFW                                                    & 319.006   & 119.014       \\ \cline{2-4}
		\multirow{-2}{*}{w8a}          & DstoFW                                                   & 47.281    & 24.288        \\ \hline
		& DenFW                                                    & 1206.455  & 1084.169      \\ \cline{2-4}
		\multirow{-2}{*}{covtype}      & DstoFW                                                   & 104.154   & 47.611        \\ \hline
	\end{tabular}
	
\end{table}
 \par To further test the computation efficiency, we provide the execution times of distributed algorithms DenFW and DstoFW for solving the convex and non-convex optimization \eqref{opti_p} in Table \ref{exec_tab}. All distributed algorithms are applied over the multi-agent network in Fig. \ref{topo_fig} and implemented in Python over the MPI distributed model, which is a high-performance message passing interface. This multi-processor environment is based on one computer with a Core(TM) I5-8250U CPU, 1.6GHz. By comparison, we observe that the proposed stochastic DstoFW algorithm costs less execution time than DenFW for both convex and non-convex optimization, which benefits from the fewer communication rounds with neighbors and the lower computational cost of stochastic gradients. In addition, with the size of dataset increases, the superiority of the execution time of DstoFW algorithm becomes more significant. It verifies the effectiveness of the introduced stochastic technique for large-scale problems.

\section{Conclusion}\label{conclusion}
Focusing on large-scale constrained finite-sum optimization, this paper provided one distributed stochastic Frank-Wolfe algorithm. By combining gradient tracking and variance reduction technique, the proposed algorithm deals with the local and global variances caused by random sampling and distributed data over multi-agent networks. For convex and non-convex optimization problems, the proposed stochastic algorithm converges at the rate of $\mathcal{O}(1/k)$ and $\mathcal{O}(\frac{1}{K^{\min\{\alpha,1-\alpha\}}})$, respectively. By comparative simulations, the proposed algorithm shows an excellent convergence performance and costs less execution time than the distributed algorithm DenFW. One future research direction is to design a distributed stochastic projection-free algorithm with a faster convergence rate by utilizing momentum or Nesterov's accelerated technique.
\section{appendix}\label{append}
\subsection{Proof of Lemma \ref{maxbarxisub}}\label{proof_lem}
\begin{proof}
	For convenience, we drop the dependence of $\alpha$ in $k_0(\alpha)$. In addition, note that $\max_{i\in[m]}\|\overline{x_i^k}-\bar{x}^k\|\leq \sqrt{\sum_{i=1}^m \|\overline{x_i^k}-\bar{x}^k\|^2}$. To prove Lemma \ref{maxbarxisub}, we show that for all $k\geq 1$,
	\begin{align}\label{in_prof}
		\sqrt{\sum_{i=1}^m \|\overline{x_i^k}-\bar{x}^k\|^2}\leq C_p \gamma_k, \ C_p=k_0^{\alpha} \sqrt{m}\bar{\rho},
	\end{align}
	where $\gamma_k=\frac{1}{k^{\alpha}}$.
	We prove this lemma using the induction hypothesis. 
	\par Because $\overline{x_i^k}, \ \bar{x}^k\in \Omega$ and the diameter of $\Omega$ is bounded by $\bar{\rho}$, \eqref{in_prof} holds for $1\leq k\leq k_0$.
	\par Next, suppose that $\sqrt{\sum_{i=1}^m \|\overline{x_i^k}-\bar{x}^k\|^2}\leq C_p\gamma_{k}$ holds for some $k\geq k_0$. We show that $\sqrt{\sum_{i=1}^m \|\overline{x_i^{k+1}}-\bar{x}^{k+1}\|^2} \leq \gamma_{k+1}C_p$ holds. 
	\par Recall that $x_i^{k+1}=(1-\gamma_k)\overline{x_i^k}+\gamma_k u_i^k$. Let $\bar{u}^k\triangleq \frac{1}{m}\sum_{i=1}^m u_i^k$, then
	\begin{align}\label{barxk} \bar{x}^{k+1}=(1-\gamma_k)\bar{x}^k+\gamma_k \bar{u}^k.
	\end{align}
	By \eqref{barxk} and Proposition \ref{barsumpro}, we obtain
	\begin{align}\label{sumkplus}
		&\sum_{i=1}^m \|\overline{x_i^{k+1}}-\bar{x}^{k+1}\|^2\notag\\
		\leq &|\lambda_2(W)|^2 \sum_{i=1}^m \|(1-\gamma_k)(\overline{x_i^k}-\bar{x}^k)+\gamma_k(u_i^k-\bar{u}^k)\|^2.
	\end{align}
	The right-hand side of \eqref{sumkplus} satisfies
	\begin{align}\label{sumxx}
		&\sum_{i=1}^m \|(1-\gamma_k)(\overline{x_i^k}-\bar{x}^k)+\gamma_k(u_i^k-\bar{u}^k)\|^2\notag\\
		\leq & \sum_{i=1}^m \Big(\|\overline{x_i^k}-\bar{x}^k\|^2+\gamma_k^2\bar{\rho}^2+2\bar{\rho}\gamma_k\|\overline{x_i^k}-\bar{x}^k\|\Big)\notag\\
		\leq& \gamma_k^2(C_p^2+m\bar{\rho}^2)+2\bar{\rho}\gamma_k \sqrt{m}\sqrt{\sum_{i=1}^m\|\overline{x_i^k}-\bar{x}^k\|^2}\notag\\
		= &\gamma_k^2(C_p+\sqrt{m}\bar{\rho})^2\notag\\
		\leq &(\frac{(k_0^{\alpha}+1)}{k_0^{\alpha}k^{\alpha}}C_p)^2,
	\end{align}
	where the second inequality holds due to $\sum_{i=1}^m |c_i| \leq \sqrt{m}\sqrt{\sum_{i=1}^m c_i^2}$ and the induction hypothesis. What's more, it follows from \eqref{k0_def} that for all $k\geq k_0$,
	\begin{align}\label{lambdaww}
		|\lambda_2(W)|\frac{(k_0^{\alpha}+1)}{k_0^{\alpha}k^{\alpha}}\leq & (\frac{k_0}{k_0+1})^{\alpha}\frac{k_0^{\alpha}}{k_0^{\alpha}+1} \frac{(k_0^{\alpha}+1)}{k_0^{\alpha}k^{\alpha}}\notag\\
		\leq& \frac{k^{\alpha}}{(k+1)^{\alpha}}\frac{1}{k^{\alpha}}\notag\\
		= & \frac{1}{(k+1)^{\alpha}},
	\end{align}
	where we use the monotonically increase property of function $g(v) = (v/(1 + v))^{\alpha}$ with respect to $v$ in the second inequality. 
	Substituting \eqref{sumxx} and \eqref{lambdaww} to \eqref{sumkplus}, we obtain the desired result
	\begin{align*}
		\sqrt{\sum_{i=1}^m \|\overline{x_i^{k+1}}-\bar{x}^{k+1}\|^2} \leq \frac{1}{(k+1)^{\alpha}}C_p=\gamma_{k+1}C_p.
	\end{align*}
\end{proof}
\subsection{Proof of Proposition \ref{usubx_pro}}\label{proof_pro}
\begin{proof}
		1) At first, taking the average of $d_i^{k}$ in \eqref{line_t}, $\bar{d}^{k}=\bar{g}^k$ holds due to the double stochasticity of $W$. Then, taking the average of $g_i^k$ in \eqref{line_g},
	\begin{align}\label{bardgv}
		\bar{d}^k=\bar{g}^k&=\frac{1}{m}\sum_{i=1}^m(d_i^{k-1}+v_i^k-v_i^{k-1})\notag\\
		&=\bar{d}^{k-1}+\frac{1}{m}\sum_{i=1}^m (v_i^k-v_i^{k-1})\notag\\
		&=\bar{d}^1+\frac{1}{m}\sum_{i=1}^m \sum_{\tau=2}^k (v_i^{\tau}-v_i^{\tau-1})\notag\\
		&=\bar{v}^k,
	\end{align}
	where the second-to-last equation holds by induction, and the last equality holds due to $d_i^1=g_i^1=v_i^1$ for all $i \in [m]$.
	\par 2)
Because of $\bar{d}^k=\bar{v}^k$ in \eqref{bardgv}, 
	\begin{align*}
		&\|\bar{d}^k-\overline{\nabla_k F}\|\\
		=&\|\bar{v}^k-\overline{\nabla_k F}\|\\
		=&\|\frac{1}{m}\sum_{i=1}^m v_i^k-\frac{1}{m}\sum_{i=1}^m \nabla f_i({x_i^k})\|\\
		\leq & \frac{1}{m}\sum_{i=1}^m \|v_i^k-\nabla f_i({x_i^k})\|.
	\end{align*}
	We take expectation of the above inequality and get
	\begin{align}\label{bardnablaF}
		\mathbb{E}[\|\bar{d}^k-\overline{\nabla_k F}\|]\leq \frac{1}{m}\sum_{i=1}^m \mathbb{E}[\|v_i^k-\nabla f_i({x_i^k})\|].
	\end{align}
	{If ${\rm mod}(k,q)=0$ (i.e., there exists $\tilde{n}_k\in \mathbb{Z}^+$ such that $k=\tilde{n}_kq$),  $v_i^{k}=\nabla f_i({x_{i}^{k}})$ holds by \eqref{mod_up} and $\mathbb{E}[\|\bar{d}^k-\overline{\nabla_k F}\|]=0$ holds in \eqref{bardnablaF}.\\
		If $(\tilde{n}_k-1)q+1\leq k\leq \tilde{n}_kq-1$, it follows from \cite[Lemma 1]{pmlr-v97-yurtsever19b} that
		\begin{align}\label{nabfv}
			&\mathbb{E}_k[\|\nabla f_i({x_{i}^{k}})-v_i^{k}\|^2]\notag\\
			\leq& \frac{L^2}{|S^{k-1}|}\|{x_{i}^{k}}-{x_i^{k-1}}\|^2+\|\nabla f_i({x_i^{k-1}})-v_i^{k-1}\|^2.
		\end{align}
		Recall that $\max_{i\in[m]} \|\overline{x_i^k}-\bar{x}^k\|\leq C_p\gamma_k$ in \eqref{xsubbax}, and 	
		$\|x_i^{k}-\bar{x}^{k}\|\leq (2\bar{\rho}+3C_p)2^{\alpha}\gamma_{k}$ in \eqref{xixsub}.
Combining \eqref{xsubbax} and \eqref{xixsub}, we have
		\begin{align}\label{xikbar}
			\|\overline{x_i^{k}}-x_i^{k}\| \leq& \|\overline{x_i^{k}}-\bar{x}^{k}\|+\|\bar{x}^{k}-x_i^{k}\|\notag\\
			{\leq} & C_p\gamma_{k}+ (2\bar{\rho}+3C_p)2^{\alpha}\gamma_{k}\notag\\
			\overset{\alpha\leq 1}{\leq} & (4\bar{\rho}+7C_p)\gamma_{k}.
		\end{align}
	}
	
		Then, the term $\|x_i^{k}-x_i^{k-1}\|$ in the right-hand side of \eqref{nabfv} satisfies
		\begin{align}\label{xksubkk}
			\|x_i^{k}-x_i^{k-1}\|\leq& \|\overline{x_i^{k-1}}-x_i^{k-1}\|+\gamma_{k-1}\|u_i^{k-1}-\overline{x_i^{k-1}}\|\notag\\
			\leq& (4\bar{\rho}+7C_p)\gamma_{k-1}+\gamma_{k-1}\bar{\rho} \notag\\
			\leq& (5\bar{\rho}+7C_p)\gamma_{k-1}.
		\end{align}
		\par Hence, \eqref{nabfv} satisfies
		\begin{align}\label{nablafvi}
			&\mathbb{E}_k[\|\nabla f_i({x_{i}^{k}})-v_i^{k}\|^2]\notag\\
			\leq &\frac{L^2}{|S^{k-1}|}(5\bar{\rho}+7C_p)^2\gamma_{k-1}^2+\|\nabla f_i({x_i^{k-1}})-v_i^{k-1}\|^2.
		\end{align}
		By taking the total expectation on both sides of \eqref{nablafvi} and telescoping over $k$, we have
		\begin{align}\label{nablafv}
			\mathbb{E}[\|\nabla f_i({x_{i}^k})-v_i^{k}\|^2]\leq& \mathbb{E}[\|\nabla f_i({x_i^{\tilde{n}_kq}})-v_i^{\tilde{n}_kq}\|^2]\notag\\
			&+\sum_{t=(\tilde{n}_k-1)q+1}^{k-1} \frac{\gamma_{t}^2 L^2 (5\bar{\rho}+7C_p)^2}{|S^{t}|}\notag\\
			= &L^2 (5\bar{\rho}\!+\!7C_p)^2 \!\sum_{t=(\tilde{n}_k-1)q+1}^{k-1} \!\frac{\gamma_{t}^2 }{|S^{t}|},
		\end{align}
		where $\nabla f_i({x_i^{\tilde{n}_kq}})=v_i^{\tilde{n}_kq}$ holds in the first inequality due to \eqref{mod_up}.
		Because $(\mathbb{E}X)^2\leq \mathbb{E}X^2$ holds for random variable $X$, $\mathbb{E}[\|\nabla f_i({x_{i}^k})-v_i^{k}\|]\leq \sqrt{\mathbb{E}[\|\nabla f_i({x_{i}^k})-v_i^{k}\|^2]}$. Substituting \eqref{nablafv} into \eqref{bardnablaF} implies
		{
			\begin{align}\label{bark2}
				\mathbb{E}[\|\bar{d}^k-\overline{\nabla_k F}\|]&\leq \frac{1}{m}\sum_{i=1}^m L (5\bar{\rho}+7C_p) \sqrt{\sum_{t=(\tilde{n}_k-1)q+1}^{k-1} \frac{\gamma_{t}^2 }{|S^{t}|}}\notag\\
				&\leq \frac{1}{m}\sum_{i=1}^m L (5\bar{\rho}+7C_p) \sum_{t=(\tilde{n}_k-1)q+1}^{k-1} \sqrt{\frac{\gamma_{t}^2 }{|S^{t}|}}\notag\\
				&\leq L (5\bar{\rho}+7C_p) \frac{q}{\sqrt{|S^{\tilde{n}_kq-1}|}} \gamma_{k-1} \notag \\
				&\leq L (5\bar{\rho}+7C_p)2^{\alpha}   \gamma_{k},
			\end{align}
			where the fourth inequality holds due to Sampling rule 1 and the last inequality holds because of $q=\sqrt{|S^{\tilde{n}_kq-1}|}$ and \eqref{gammaksubk}.}

	\par Combining the above two cases $k=\tilde{n}_k q$ and $(\tilde{n}_k-1)q+1\leq k\leq \tilde{n}_kq-1$, we obtain the desired result
	\begin{align}\label{dnablaF}
		\mathbb{E}[\|\bar{d}^k-\overline{\nabla_k F}\|]\leq L (5\bar{\rho}+7C_p)2^{\alpha} \gamma_k.
	\end{align}
	\par 3) Using \eqref{nablafv}, Sampling rule 1 and $\gamma_k<1$, we obtain
	\begin{align}\label{v1}
		\mathbb{E}[\|\nabla f_i({x_{i}^k})-v_i^{k}\|^2]\leq  L^2 (5\bar{\rho}+7C_p)^2.
	\end{align}
	Then, by Jensen's inequality,
	\begin{align}\label{vsq}
		\mathbb{E}[\|v_i^k\|^2]\leq& \mathbb{E}[2\|\nabla f_i({x_{i}^k})-v_i^{k}\|^2]+\mathbb{E}[2\|\nabla f_i(x_i^k)\|^2]\notag\\
		\leq &2 L^2 (5\bar{\rho}+7C_p)^2+2\mathbb{E}[\|\nabla f_i(x_i^k)\|^2]\notag\\
		\leq & 2L^2 (5\bar{\rho}+7C_p)^2+ 2C^2\triangleq \mathbb{C},
	\end{align}
	where the second inequality holds due to \eqref{v1} and the last inequality holds due to Fact 1. 
	\par Then, we prove \eqref{dibard} using the induction hypothesis. For convenience, we drop the dependence of $\alpha$ in $k_0(\alpha)$ in \eqref{k0_def}. At first, we prove \eqref{dibard} holds for $k\leq k_0$.
	It follows from \eqref{line_g} that
	\begin{align}\label{k0}
		&\mathbb{E}\Big[\sum_{i=1}^m \|d_i^k-\bar{v}^k\|^2\Big]\notag\\
		=&\mathbb{E}\Big[\sum_{i=1}^m \|g_i^{k+1}+v_i^k-v_i^{k+1}-\bar{v}^k\|^2\Big]\notag\\
		\leq & 4\sum_{i=1}^m\mathbb{E}[\|g_i^{k+1}\|^2]+4\sum_{i=1}^m\mathbb{E}[\|v_i^k\|^2]+4\sum_{i=1}^m\mathbb{E}[\|v_i^{k+1}\|^2]\notag\\
		&+4\sum_{i=1}^m\mathbb{E}\Big[\Big\|\frac{1}{m}\sum_{j=1}^m v_j^k\Big\|^2\Big]\notag\\
		\leq & 4\sum_{i=1}^m\mathbb{E}[\|g_i^{k+1}\|^2]+12m\mathbb{C},
	\end{align}
	where the last inequality holds due to \eqref{vsq}.
	In addition, the first term of right-hand side of \eqref{k0} satisfies
	\begin{align}\label{gnorm}
		&\mathbb{E}[\|g_i^{k+1}\|^2]\notag\\
		=&\mathbb{E}\Big[\Big\|\sum_{j=1}^m W_{ij}g_j^k+v_i^{k+1}-v_i^k\Big\|^2\Big]\notag\\
		\leq & 3\mathbb{E}\Big[\Big\|\sum_{j=1}^m W_{ij}g_j^k\Big\|^2+\|v_i^{k+1}\|^2+\|v_i^k\|^2\Big]\notag\\
		\leq &3\mathbb{E}\Big[\Big\|\sum_{j=1}^m W_{ij}g_j^k\Big\|^2\Big]+6\mathbb{C}\notag\\
		\leq & 3\sum_{j=1}^m W_{ij} \mathbb{E}[\|g_j^k\|^2]+6\mathbb{C}\notag\\
		\leq &3^k C^2+6\sum_{s=0}^{k-1}3^s\mathbb{C}\notag\\
		\leq &3^k C^2+\frac{6 \mathbb{C} }{1-3}(1-3^k)\notag\\
		\leq &3^{k_0} C^2+3 \mathbb{C}(3^{k_0}-1) ,
	\end{align}
	where the second inequality holds due to \eqref{vsq}, the third inequality holds by the fact that $W$ is doubly stochastic and the operator $\|\cdot\|^2$ is convex, and the fourth inequality holds due to $\|g_j(1)\|^2=\|\nabla f_j(x_j(1))\|^2\leq C^2, \ \forall j\in [m]$.
	Then, substituting \eqref{gnorm} into \eqref{k0} yields
	\begin{align*}
		&\mathbb{E}\Big[\sum_{i=1}^m \|d_i^k-\bar{v}^k\|^2\Big]\\
		\leq & 4m3^{k_0} C^2+12m3^{k_0}\mathbb{C}-12m \mathbb{C}+12m\mathbb{C}\\
		= & 4m 3^{k_0}C^2+12m3^{k_0}\mathbb{C}.
	\end{align*}
	In addition, $\bar{d}^k=\bar{v}^k$  holds in \eqref{bardgv}.
	Hence, for $1\leq k\leq k_0$, $\mathbb{E}[\sum_{i=1}^m \|d_i^k-\bar{d}^k\|^2]\leq  C_g^2\gamma_k^2$, where $\gamma_k=\frac{1}{k^{\alpha}}$ and $C_g=k_0^{\alpha}\Big(2mL(5\bar{\rho}+7C_p)+\big(4m3^{k_0}C^2+12m3^{k_0}\mathbb{C}\big)^{1/2}\Big)$.
	\par Then, we assume that $\mathbb{E}[\sum_{i=1}^m \|d_i^k-\bar{d}^k\|^2]\leq  C_g^2\gamma_k^2$ holds for some $k\geq k_0$. We will show that $\mathbb{E}[\sum_{i=1}^m \|d_i^{k+1}-\bar{d}^{k+1}\|^2]\leq  C_g^2\gamma_{k+1}^2$ holds. Recall that $d_i^{k+1}=\sum_{j=1}^m W_{ij} g_j^{k+1}$ and $\bar{d}^{k+1}=\frac{1}{m}\sum_{i=1}^m d_i^{k+1}$. By Proposition \ref{barsumpro},
	\begin{align}\label{sqrtdd}
		\sqrt{\sum_{i=1}^m\|d_i^{k+1}-\bar{d}^{k+1}\|^2}\leq\! |\lambda_2(W)|\sqrt{\sum_{i=1}^m \|g_i^{k+1}-\bar{d}^{k+1}\|^2}.
	\end{align}
	Let $\sigma_i^{k}$ denote $(v_i^{k+1}-v_i^k)$ and $\bar{\sigma}^{k}$ denote $(\bar{v}^{k+1}-\bar{v}^k)$. Then, $g_i^{k+1}=d_i^k+\sigma_i^{k}$ and $\bar{d}^{k+1}=\bar{d}^k+\bar{\sigma}^{k}$ hold by \eqref{line_g} and \eqref{line_t}. 
	%
	In addition, utilizing the $L$-smoothness of cost function  and \eqref{xksubkk} , we obtain
	\begin{align*}
		\|\sigma_i^{k}\|\leq &\frac{1}{|S^k|}\sum_{j\in S^k} \|\nabla f_{i,j}(x_i^{k+1})-\nabla f_{i,j}(x_i^k)\|\\
		\leq & L (5\bar{\rho}+7C_p)\gamma_k.
	\end{align*}
	Making use of the triangle inequality, we also have
	\begin{align}\label{sigmat}
		\|\sigma_i^{k}-\bar{\sigma}^{k}\|\leq& (1-\frac{1}{m})\|\sigma_i^{k}\|+\frac{1}{m}\sum_{j\neq i}\|\sigma_j^{k}\|\notag\\
		\leq & 2 L (5\bar{\rho}+7C_p)\gamma_k.
	\end{align}
	Then, with \eqref{sqrtdd},
	\begin{align*}
		&\mathbb{E}\Big[\sum_{i=1}^m\|d_i^{k+1}-\bar{d}^{k+1}\|^2\Big]\notag\\
		\leq& |\lambda_2(W)|^2 \mathbb{E}\Big[\sum_{i=1}^m \|g_i^{k+1}-\bar{d}^{k+1}\|^2\Big]\notag\\
		= & |\lambda_2(W)|^2 \mathbb{E}\Big[\sum_{i=1}^m \|d_i^k+\sigma_i^{k}-(\bar{d}^k+\bar{\sigma}^{k})\|^2\Big]\notag\\
		\leq & |\lambda_2(W)|^2 \mathbb{E}\Big[\sum_{i=1}^m \big( \|d_i^k-\bar{d}^k\|^2+\|\sigma_i^{k}-\bar{\sigma}^{k}\|^2\\
		&+2\|d_i^k-\bar{d}^k\|\|\sigma_i^{k}-\bar{\sigma}^{k}\|\big)\Big]\notag\\
		\overset{\eqref{sigmat}}{\leq} & |\lambda_2(W)|^2\Big( C_g^2{\gamma_k}^2+m(2 L (5\bar{\rho}+7C_p){\gamma_k})^2\\
		&+2\sum_{i=1}^m\mathbb{E}[\|d_i^k-\bar{d}^k\|\|\sigma_i^k-\bar{\sigma}^k\|]\Big)\\
		\leq &|\lambda_2(W)|^2\Big(C_g^2{\gamma_k}^2+m(2 L (5\bar{\rho}+7C_p){\gamma_k})^2\\
		&+2\sum_{i=1}^m(\mathbb{E}[\|d_i^k-\bar{d}^k\|^2])^{1/2}(\mathbb{E}[\|\sigma_i^k-\bar{\sigma}^k\|^2])^{1/2}\Big)\\
		\leq &|\lambda_2(W)|^2\Big(C_g^2{\gamma_k}^2+m(2 L (5\bar{\rho}+7C_p){\gamma_k})^2\\
		&+4 L (5\bar{\rho}+7C_p){\gamma_k}\sum_{i=1}^m\sqrt{\mathbb{E}[\sum_{i=1}^m\|d_i^k-\bar{d}^k\|^2]}\Big),
	\end{align*}
where in the fourth inequality we use  the H$\ddot{o}$lder’s inequality $\mathbb{E}[|XY|]\leq (\mathbb{E}[|X|^2])^{1/2}(\mathbb{E}[|Y|^2])^{1/2}$. What's more, because of $m\geq 1$,
	\begin{align*}
		&\mathbb{E}\Big[\sum_{i=1}^m\|d_i^{k+1}-\bar{d}^{k+1}\|^2\Big]\\
	{\leq } &|\lambda_2(W)|^2\Big(C_g^2{\gamma_k}^2+m^2(2 L (5\bar{\rho}+7C_p){\gamma_k})^2\\
		&+4 m L (5\bar{\rho}+7C_p){\gamma_k}C_g\gamma_k\Big)\\
		=&|\lambda_2(W)|^2{\gamma_k}^2(C_g+2mL(5\bar{\rho}+7C_p))^2\\
		\leq &|\lambda_2(W)|^2{\gamma_k}^2(C_g+C_gk_0^{-\alpha})^2\\
		\leq& |\lambda_2(W)|^2\Big(\frac{C_g(k_0^{\alpha}+1)}{k_0^{\alpha}(k)^{\alpha}}\Big)^2.
	\end{align*}
	Finally, because \eqref{lambdaww} holds for $k\geq k_0$, 
	\begin{align*}
		\mathbb{E}\Big[\sum_{i=1}^m\|d_i^{k+1}-\bar{d}^{k+1}\|^2\Big]
		\leq & C_g^2(\frac{1}{(k+1)^{\alpha}})^2=C_g^2\gamma_{k+1}^2.
	\end{align*}
\end{proof}
\bibliographystyle{ieeetran}
\bibliography{refer}
\end{document}